\documentclass[onefignum,onetabnum]{siamart190516}
\usepackage{epsfig,amssymb,hyperref,enumitem, ulem, stackengine, pict2e, wrapfig}
\stackMath
\usepackage{bm,bbm,mathtools}
\usepackage{caption}
\usepackage{subcaption}
\usepackage{todonotes}
\usepackage{multirow, tabularx, placeins}
\usepackage{etoolbox}

\usepackage[algo2e]{algorithm2e}
\SetKwInput{KwInput}{Input} 
\SetKwInput{KwOutput}{Output} 

\setlist[enumerate]{itemsep=5pt, topsep=5pt,after=\vspace{3pt}}
\setlist[itemize]{itemsep=5pt, topsep=5pt,after=\vspace{3pt}}

\newsiamthm{claim}{Claim}
\newsiamthm{defn}{Definition}
\newsiamthm{remark}{Remark}

\renewcommand{\l}{\left}
\renewcommand{\r}{\right}

\newcommand{\mbf}{\mathbf}

\newcommand{\bz}{{\bf z}}
\newcommand{\bh}{{\bf h}}
\newcommand{\bv}{{\bf v}}
\newcommand{\bw}{{\bf w}}
\newcommand{\arr}{{\text{arr}}}

\newcommand{\E}{\mathbb E\,}
\newcommand{\Z}{\mathbb Z}
\newcommand{\ol}{\overline}
\newcommand{\bml}{{\bm\lambda}}

\newcommand{\calZ}{\mathcal Z}

\newcommand{\R}{\mathbb R}
\newcommand{\p}{\mathbb P}

\renewcommand{\L}{\mathcal L}

\newcommand{\J}{\mathcal J}

\newcommand{\e}{\,\mathrm{exp}}

\newcommand{\Law}{{\mathrm{Law}}}
\newcommand{\gibbs}{{\text{gibbs}}}
\newcommand{\hatJgibbs}{\hat J_\gibbs}
\newcommand{\hatJ}{\hat J}
\newcommand{\rhogibbsNd}{\rho[\bml]}

\newcommand{\eV}{\eqref{V}}
\newcommand{\eE}{\eqref{E}}
\newcommand{\eEp}{\eqref{Ep}}
\newcommand{\eEf}{\eqref{Ef}}
\newcommand{\eEfp}{\eqref{Efp}}

\newcommand{\beq}{\begin{equation}}
\newcommand{\eeq}{\end{equation}}
\newcommand{\beqn}{\begin{equation*}}
\newcommand{\eeqn}{\end{equation*}}
\def\beqs#1\eeqs{%
    \begin{equation}\begin{split}%
    #1%
    \end{split}\end{equation}%
}
\def\beqsn#1\eeqsn{%
    \begin{equation*}\begin{split}%
    #1%
    \end{split}\end{equation*}%
}
\newcommand{\unit}{{\mathbb T}}
\newcommand{\cts}{C(\unit)}
\newcommand{\bump}{C^\infty(\unit)}
\newcommand{\idxsetx}{{N(x\pm\epsilon)}}

\newcommand{\argidxsetx}[1]{{#1\in \idxsetx}}

\newcommand{\varbar}[1]{\ol{\bf #1}_\idxsetx}
\newcommand{\fxnbar}[2]{\bar #1({\bf #2}_\idxsetx)}

\newcommand{\epsDel}{{\epsilon,\Delta}}
\newcommand{\Nepslim}{N\to\infty,\,\epsilon\to0}
\newcommand{\limNeps}{\lim_{\epsilon\to0}\lim_{N\to\infty}}
\newcommand{\comment}[1]{\textcolor{red}{[#1]}}
\newcommand{\vN}{\mathbf v_N}
\newcommand{\wN}{\mathbf w_N}
\newcommand{\zN}{\mathbf z_N}
\newcommand{\hN}{\mathbf h_N}

\newcommand{\LN}{\L_N}

\newcommand{\probto}{\stackrel{\text{in prob}}{\to}}

\title{From local equilibrium to numerical PDE: Metropolis crystal surface dynamics in the rough scaling limit\thanks{\textbf{Funding}: This material is based upon work supported by U.S. Department of Energy, Office of Science, Office of Advanced Scientific Computing Research, Department of Energy Computational Science Graduate Fellowship under Award Number DE-FG02-97ER25308. The author was also supported in part by the Research Training Group in Modeling and Simulation funded by the National Science Foundation via grant RTG/DMS – 1646339}}
\author{Anya Katsevich\thanks{Department of Mathematics, Courant Institute of Mathematical Sciences, New York University. \email{katsevich@cims.nyu.edu}}}
\begin{document}

\normalem
\maketitle
\begin{abstract}
We derive the PDE governing the hydrodynamic limit of a Metropolis rate crystal surface height process in the ``rough scaling" regime introduced by Marzuola and Weare. The PDE takes the form of a continuity equation, and the expression for the current involves a numerically computed multiplicative correction term similar to a mobility. The correction accounts for the fact that, unusually, the local equilibrium distribution of the process is not a local Gibbs measure even though the global equilibrium distribution \emph{is} Gibbs. We give definitive numerical evidence of this fact, originally suggested in Gao, et. al., \emph{Pure and Applied Analysis} (2021). In that paper, an approximate PDE --- our PDE, but without the correction term --- was derived for the limit of the Metropolis rate process under the assumption of a local Gibbs distribution. Our main contribution is to present a numerical method to compute the corrected macroscopic current, which is given by a function of the third spatial derivative of the height profile. Our method exploits properties of the local equilibrium (LE) state of the third order finite difference process. We find that the LE state of this process is not only useful for deriving the PDE; it also enjoys nonstandard properties which are interesting in their own right. Namely, we demonstrate that the LE state is a ``rough LE", a novel kind of LE state discovered in our recent work on an Arrhenius rate crystal surface process. \end{abstract}


\section{Introduction}
%
%
%
%
%
%
%
%
Consider a microscopic particle system in global equilibrium. Taking the viewpoint of equilibrium statistical mechanics, we can describe the system by its ensemble, a probability distribution over particle configurations. Typically, physical principles dictate that the distribution belong to some family of distributions, and a few average statistics of the system (e.g. the mean) determine the particular distribution in this family. For example, the speeds of particles in an idealized gas follow a distribution in the one-parameter family of Maxwell-Boltzmann distributions~\cite{statphys}. The average speed of the particles in the gas determines the parameter. Now consider an out-of-equilibrium particle system, evolving in time toward its global equilibrium state. If the system is \emph{locally} equilibrated, then an analogous principle applies. There is a single family of probability distributions governing the particle configurations in each space-time region of \emph{mesoscopic} extent, an intermediate scale between micro- and macroscopic. For example, the speeds of particles in a mesoscopic region of a \emph{locally} equilibrated gas can still be expected to be Maxwell-Boltzmann distributed. But unlike globally equilibrated systems, the average statistics determining the parameter now vary among these mesoscopic regions, and they also vary in time. 

If we rescale time and space appropriately, a macroscopic equation of motion ---  a partial differential equation --- emerges from the microscopic particle dynamics. The PDE governs the evolution in time of a limit of these local mesoscopic statistics, as they tend toward their single constant value in global equilibrium. Typically, knowing the parameterized family of local equilibrium (LE) distributions is sufficient to determine the PDE. 

In this paper, we derive the PDE limit of a stochastic microscopic dynamics modeling particle diffusion on a crystal surface. The global equilibrium (GE) family for this particle system are the Gibbs measures. Unusually, however, the LE family is \emph{not} made up of local Gibbs measures, as we show numerically. In other words, the LE family is \emph{not} the same as the GE family. Moreover, the LE family is not known explicitly at all. As such, it is impossible to obtain an exact analytic expression for the PDE, and we opt for a numerical approach instead. Our approach to derive the PDE exploits fundamental properties of LE states without needing to know the LE family explicitly. Although understanding why the LE family is not local Gibbs is an interesting and important problem, it is beyond the scope of this paper. 

\subsection{Background and Main Contribution} We now give some background on the crystal surface model and on related works. We model the crystal surface as a collection of particles arranged in a height profile on a one-dimensional periodic lattice. At lattice site $i$, the height $h_i$ represents the number of particles which are stacked in a column above (positive ``height'')  or below (negative ``height'') the lattice, which represents height zero. 
 The particle dynamics is governed by a Markov jump process $\hN(t)\in\Z^N$ (where $N$ is the lattice size), in which the topmost particles jump to neighboring columns with certain jump rates. Jumps which lower the surface energy have higher rates than jumps which increase it, where the energetically optimal configuration is a flat surface. In this context, the micro-to-macro limit is known as a ``hydrodynamic" limit, obtained by scaling height, time, and lattice width with $N$ according to a certain scaling regime, and taking $N\to\infty$. The limit is a macroscopic height profile $h(t,x)$, where the spatial domain is the unit torus. 

Here, we assume the microscopic dynamics evolves under ``Metropolis-type'' jump rates, which are functions only of the difference in the surface energy before and after the jump. We study the macroscopic limit in a nonstandard, so-called ``rough'' scaling regime. 
The rough scaling regime was introduced in~\cite{mw-krug} to study the limit of the better-known, \emph{Arrhenius} jump rate crystal surface dynamics. Marzuola and Weare show that the PDE of the Arrhenius process in this rough scaling limit takes the form $h_t = \partial_{xx}\e(-h_{xx})$. Meanwhile, the PDE governing the more standard scaling limit (which the authors call the ``smooth scaling regime") is essentially given by linearizing the exponential in the rough PDE. Thus, the rough PDE describes surfaces $h$ in which $|h_{xx}|$ is large (and cannot be linearized), so that $h$ is ``rapidly varying" and hence ``rough". For a discussion of the physical relevance of the rough scaling regime, see e.g.~\cite{gao2020_arrPDE,asymmetry,mw-krug}. 

Similarly, the rough scaling limit for the Metropolis rate process is the limit which leads to a PDE with exponential nonlinearity. The Metropolis rough scaling limit was first studied in~\cite{gao2020}. In this work, the authors assume the LE distribution can be approximated by a local Gibbs measure to derive the approximate PDE
\beq\label{introPDE-gibbs}h_t = -\partial_x\l(\sinh(Kh_{xxx})\r),\eeq where $K$ is inverse temperature. That the PDE takes the form of a continuity equation naturally follows from the microscopic dynamics, which preserves total sum of heights (see~\eqref{evoln-h} for the corresponding microscopic continuity equation). However, the current $\hatJgibbs(h_{xxx}):=\sinh(Kh_{xxx})$ is not the true macroscopic current. Indeed, the authors of~\cite{gao2020} observe a discrepancy between the solution to the PDE~\eqref{introPDE-gibbs} and the microscopic process $\hN$, which does not vanish as one increases $N$. 

Our main contribution in this paper is a numerical method which removes this discrepancy. Namely, we compute a multiplicative correction $\sigma$ to the current, to obtain the true current $\hat J=\sigma\times\hatJgibbs$, and the true PDE
\beq\label{introPDE}h_t = -\partial_x\l(\sigma(h_{xxx})\sinh(Kh_{xxx})\r).\eeq
The method uses sample runs of the microscopic process generated from only a single initial datum, and requires evolving the process in time only sufficiently long to reach local, rather than global, equilibrium. The function $\sigma$ is $K$-dependent and converges to $1$ as $K\downarrow0$. This shows that the local Gibbs approximation becomes accurate in the small $K$ limit. Using observed qualitative properties of $\sigma$ (e.g. that it is even and increasing when $h_{xxx}>0$), we also extend the results of~\cite{gao2020} on properties of the PDE~\eqref{introPDE-gibbs}. Namely, we show that strong solutions of~\eqref{introPDE} exist, are unique, and enjoy the same regularity properties as those shown for solutions to~\eqref{introPDE-gibbs}. 

The way $\sigma$ appears in~\eqref{introPDE} bears some resemblance to a mobility: a medium-dependent constant of proportionality determining the current in a diffusion. The similarity is somewhat superficial, however, because the PDE~\eqref{introPDE} is not a standard diffusion. Indeed, the current does not follow Fick's law, since it is not proportional to the gradient of an appropriate potential. Nevertheless, we mention this similarity because like our correction $\sigma$,  mobilities arising in the continuum limit of microscopic processes often cannot be computed explicitly. There is a vast body of work on computing mobilities numerically, and we will not attempt to review it here. The closest such work to ours that we are aware of, in terms of similarity of the physical model, is~\cite{krug1995adatom}. The authors study the smooth scaling PDE limit of a microscopic crystal surface jump process, in which the rates are also of Metropolis type. The PDE limit takes the form of a standard diffusion, which allows the authors to use linear response theory to compute the slope-dependent mobility. In addition, we mention the work~\cite{Embacher2018} (see also the references therein). This work is similar to ours in that the authors' approach to compute the mobility only requires simulating the process until local equilibrium.

Since the PDE~\eqref{introPDE} is not a standard diffusion, methods for computing mobilities such as linear response are not available to us in computing the factor $\sigma$. Instead we develop an alternative numerical approach. It is borne out of the LE framework of our recent work~\cite{kat21} on rough-scaled processes, as we now explain. 

\subsection{Companion Finite Difference Process and Rough LE} Note that the macroscopic current in the PDE~\eqref{introPDE} is a function of $h_{xxx}$. This is a reflection of the fact that (1) the jump rates are functions of the third order finite difference $w_i:=h_{i+2}-3h_{i+1}+3h_i-h_{i-1}$, and (2) in the rough scaling regime, $w_i$ has order $O(1)$ as $N\to\infty$. This consideration motivates us to consider the companion process $\wN(t)=(w_1(t),\dots, w_N(t))$ --- in particular, the distribution $\Law(\wN(t))$ in local equilibrium --- as the central object of study. 
A finite difference (FD) process also plays a central role in our previous work~\cite{kat21}, in which we take a closer look at the Arrhenius rate process in the rough scaling regime. There, we show that the PDE governing the hydrodynamic limit $h$ is determined by the LE distribution of the second order FDs of the heights. We will call this second order FD process $\wN^\arr$, for comparison with the third order FD $\wN$ of the Metropolis height process.

We show in~\cite{kat21} that $\wN^\arr$ has a novel, ``rough'' LE state. The defining characteristic of the rough LE state is that the expected profile $(\E w_i^\arr)_{i=1}^N$ is rough in the sense that $|\E w_{i+1}^\arr - \E w_i^\arr|$ does not go to zero as $N$ increases. (The discovery of this rough profile retroactively lends a second meaning to the name ``rough scaling regime", which was coined earlier for different reasons). Moreover, the distributions $\Law(w_i)$ do not enjoy a crucially important feature enjoyed by more standard particle systems: belonging to a mean-parameterized measure family. However, we show that the probability distributions given by mesoscopic window averages of the single site marginals \emph{do} have this property, and their means \emph{do} vary smoothly across space. 

For the Metropolis process, we do not have explicit access to $\Law(\wN)$. However, we will show empirically that $\wN$ also has a rough LE state, confirming that this new kind of LE is not an isolated phenomenon. Building off the work in~\cite{kat21}, our numerical method for computing $\sigma$ exploits the crucial fact that upon mesoscopic window averaging, the LE state is described by \emph{some} mean-parameterized family. We will not need to know \emph{which} family this is. 

The function $\hat J=\sigma\times\hatJgibbs$ will be defined in terms of properties of the LE state of $\wN$. To show this same $\hat J$ is the macroscopic current arising in the $h$ PDE, we take two more steps. First, we prove that if $\wN$ converges to a macroscopic $w$ in an appropriate scaling regime, then $w$ must be the solution to $w_t=-\partial_{xxxx}\hat J(w)$. Second, we prove that $\hN$ then has a unique limit $h$ in the rough scaling regime, where $h_{xxx}=w$ and $h$ is the solution to $h_t=-\partial_x\hat J(h_{xxx})$. Our proofs rely on two boundedness conditions which we confirm numerically,  but are otherwise rigorous. These two steps were also informally described in~\cite{kat21} (in particular we did not check the boundedness conditions), but they served only as motivation for studying the LE properties in that paper. 


\subsection*{Organization}The rest of the paper is organized as follows. In Section~\ref{sec:micro}, we introduce the Metropolis height process as well as the companion finite difference processes. In Section~\ref{sec:hydro}, we define the hydrodynamic limit in the rough scaling regime and motivate studying the limit of $\hN$ via the limit of the third order FDs $\wN$. In Section~\ref{sec:theory}, we formalize this approach, proving that the limit $h$ of $\hN$ and the PDE governing it follow from the limit $w$ of $\wN$ and the corresponding PDE. We also introduce a key property of rough LE states which makes our numerical method possible. In Section~\ref{sec:LE}, we review the concept of LE states, show $\wN$ has a rough LE state, and explain how the macroscopic current $\hat J$ arises from LE properties. We also show the local Gibbs measure is not correct, so that we do not know the explicit form of the LE state and cannot compute $\hat J$ analytically. In Section~\ref{sec:num} we present our numerical method and confirm that we have derived the correct PDE for $h$. Finally, we analyze the PDE in Section~\ref{sec:PDE}, and make a few concluding remarks in Section~\ref{sec:conclude}.

\subsection*{Notation} For a sequence of vectors $\bv_N\in\R^N$, $N=1,2,\dots$, we let the entries of $\bv_N$ be $\bv_N=(v_1,v_2,\dots, v_N)$, omitting the dependence of each $v_i$ on $N$ for brevity. We let $\unit$ denote the unit interval with periodic boundary conditions (the unit torus). 
Next, let $m_1(\rho)$ denote the first moment of a probability mass function (pmf) $\rho$ on $\Z$, i.e. $m_1(\rho)=\sum_{n=-\infty}^\infty n\rho(n)$, and we write $$\rho(f) :=\sum_{n=-\infty}^\infty f(n)\rho(n)$$ to denote the expectation of the observable $f$ under $\rho$. We use the notation $$\{\rho[\lambda]\mid\lambda\in\R\}$$ to denote a family of pmfs on $\Z$, parameterized by $\lambda$; so that $\rho[\lambda](n)$ denotes the probability of $n$ under $\rho[\lambda]$, and $\rho[\lambda](f)$ denotes the expectation of $f$ under $\rho[\lambda]$. 
\subsection*{Acknowledgments} Thanks to Yuan Gao, Jian-Guo Liu, Jianfeng Lu, and Jeremy Marzuola, with whom the author discussed the possibility of generalizing the analysis of~\eqref{introPDE-gibbs} to that of the PDE~\ref{introPDE}. Thanks also to Jonathan Weare and Jeremy Marzuola for their guidance and insights throughout the last five years, in which this project came to fruition. Finally, thank you to NYU High Performance Computing for access to computing resources.  

\section{Preliminaries: Metropolis Rate Model}\label{sec:micro} In this section we will introduce the Metropolis rate crystal height process $\hN$ as well as two companion processes. We will then explain the role of the companion processes in deriving the PDE governing the hydrodynamic limit of $\hN$. Finally, we will briefly mention key features of the Arrhenius rate process studied in~\cite{kat21}. This process will repeatedly serve as a point of comparison to the Metropolis process. 
\subsection{Microscopic Dynamics}\label{subsec:micro} Let $\hN(t) = (h_1(t),\dots,  h_N(t))\in\Z^N$ be a Markov jump process, with $h_j(t)$ representing the discrete height at lattice site $j$ of the crystal, relative to some fixed arbitrarily chosen zero height level. Note that each $h_i$ depends on $N$ as well as $i$. The lattice is periodic, so that we identify $j$ with $j+mN$, $m\in\Z$. As we will soon see, the dynamics of the height process $\hN(t)$ will be determined entirely by the companion \emph{slope} process $\zN(t)= (z_1(t),\dots,  z_N(t))$, where
$$z_i(t) = h_{i+1}(t) - h_i(t).$$ We will use the letters $\bh$ and $\bz$, with no subscript, to denote an arbitrary height and slope configuration in $\Z^N$, respectively. The surface energy, or Hamiltonian, of a configuration $\bh$ is given by
 \beq\label{H} H(\bh) = \sum_{i=1}^N(h_{i+1}-h_i)^2 =\sum_{i=1}^Nz_i^2=H(\mbf z).\eeq  
Although the absolute value potential is the most physically relevant choice for modeling crystal surface energies, we choose a quadratic interaction potential because certain calculations can be done explicitly in this case. In the mathematical study of hydrodynamic limits of interfaces, it is standard to consider energies of the form $\sum_iV(|h_{i+1}-h_i|)$ for general convex $V$. See e.g.~\cite{Nishikawa,funaki1997motion}. We see that the energy of a height profile $\bh$ is actually a function of the corresponding slope profile $\bz$. 

The process $\hN(t)$ evolves through particle jumps between neighboring lattice sites which, on average, lower the surface energy. We represent a jump from lattice site $i$ to site $j$ by $\mbf h\mapsto \mbf h^{i,j}$, $|i-j|=1$. Here, $\mbf h^{i,j}$ is the height profile such that
\beq\label{hij}
(h^{i,j})_k = \begin{cases}
h_i-1,\quad &k=i,\\
h_j+1,\quad &k=j,\\
h_k,\quad&\text{otherwise.}
\end{cases}
\eeq
Now, suppose $\hN(t)=\bh$, so that $\zN(t)=\bz$, the corresponding slope profile. If $\hN(t)$ undergoes the transition $\bh\mapsto\bh^{i,j}$, then $\zN(t)$ undergoes the transition $\bz\mapsto\bz^{i,j}$, where $\bz^{i,j}$ is the slope profile corresponding to $\bh^{i,j}$. Explicitly, we compute 
\beqs\label{zij}
&\mbf z^{i,i+1} = \mbf z -{\bf d}^{(2)},\quad\mbf z^{i+1,i}=\bz+{\bf d}^{(2)},\\
&{\bf d}^{(2)}=\mbf e^{i-1} - 2\mbf e^i + \mbf e^{i+1}
\eeqs where $\mbf e^j$ denotes the $j$th unit vector. The jumps $\mbf h\mapsto \mbf h^{i,j}$ occur at certain \emph{rates} $r^{i,j}_N(\mbf h)$. The rates indicate the probability of a jump in time $dt$, as follows: suppose the process is in state $\bh$ at time $t$, and let $R_N(\bh)$ be the sum of the rates of all possible jumps from $\bh$. Then the probability that the jump $\bh\mapsto\bh^{i,j}$ occurs in the interval $(t,t+dt]$ is $(r^{i,j}_N(\mbf h)/R_N(\bh))dt$. 

In this paper, we will consider rates of the form $r^{i,j}_N(\bh)=N^4r^{i,j}(\mbf h)$. Here, $N^4$ is the appropriate time scaling to take a hydrodynamic limit, as we will explain in Section~\ref{sec:hydro}. The unscaled rates $r^{i,j}(\mbf h)$ only depend on the local configuration of heights, and not on $N$. 
Formally, the rates determine the dynamics of $\hN$ through the generator $\LN$:
\beq\label{gen-gen}\l(\LN f\r)(\mbf h) = N^4\sum_{|i-j|=1}r^{i,j}(\mbf h)\l[f(\mbf h^{i,j}) - f(\mbf h)\r].\eeq 
Consider applying $\LN$ to $f=\pi_i$, where $\pi_i(\mbf h) = h_i$. Note that $h_i$ decreases by 1 if a particle at $i$ jumps to $i\pm 1$, and $h_i$ increases by 1 if a particle at $i\pm1$ jumps to $i$. As a result,
\beqs\label{LN-pi-h}(\LN \pi_i)(\mbf h) &= N^4\l[\l(r^{i-1,i}-r^{i,i-1}\r)(\mbf h)-\l(r^{i,i+1}-r^{i+1,i}\r)\r](\mbf h)\\&=N^4(J^{i-1,i}-J^{i,i+1})(\mbf h),\eeqs where $J^{i,i+1}(\mbf h)= (r^{i,i+1}-r^{i+1,i})(\mbf h)$ is the \emph{current}: the net, expected amount of mass flowing from $i$ to $i+1$ per unit of unscaled time if the process is in state $\mbf h$. By definition of the generator, we then have
\beqs\label{evoln-h}
\partial_t\E[h_i(t)] = \E\l[(\LN\pi_i)(\hN(t))\r]= N^4\E\l[(J^{i-1,i}-J^{i,i+1})(\hN(t))\r]
\eeqs
This equation is valid regardless of the specific form of $r^{i,j}(\mbf h)$. It can be thought of as a microscopic continuity equation: the change in mass (height) is given by the divergence (finite difference) of a current. We now specify the ``Metropolis-type" rates considered in this paper:
\beq\label{rates-H}
r^{i,j}(\mbf h) = \e\l(-\frac K2\l[H(\mbf z^{i,j}) - H(\mbf h)\r]\r), \quad |i-j|=1.
\eeq Here, $K=1/(k_B T)$, where $k_B$ is the Boltzmann constant and $T$ is the ambient temperature, held constant over time. See Remark~\ref{rk:metrop} for an explanation of the name ``Metropolis". Using the formulas~\eqref{H} for the Hamiltonian,~\eqref{zij} for the transitions $\bz\mapsto\bz^{i,j}$ and~\eqref{rates-H} for the Metropolis rates, we obtain the following explicit expression for the rates:
\beqs\label{met-rates-general}
r^{i,i+1}(\mbf h) &=  r^{i,i+1}(\mbf z)=\e\l(-3K+K(z_{i+1}-2z_i+z_{i-1})\r),\\
r^{i+1,i}(\mbf h) &= r^{i+1,i}(\mbf z)= \e\l(-3K-K(z_{i+1}-2z_i+z_{i-1})\r).
\eeqs
Note that these rates depend on $\bh$ only through $\bz$; in fact, only through a further finite difference (FD). This implies that $\zN(t)$ is also a Markov jump process which can exist independently of a height process: it is the process which takes jumps $\bz\mapsto\bz^{i,j}$ with rates $N^4r^{i,j}(\bz)$. The form of the rates~\eqref{met-rates-general} motivates us to also introduce the third order FD process $\wN(t) = (w_1(t),\dots, w_N(t))$, with $$w_i(t) = (z_{i-1}-2 z_i + z_{i+1})(t) = (h_{i-1}-3h_{i}+3h_{i+1} - h_{i+2})(t).$$ We let $\bw$ denote a generic third order FD profile corresponding to the generic height profile $\bh$. The process $\wN(t)$ is also a Markov jump process which can be independently defined. It undergoes jumps $\bw\mapsto \bw^{i,i+1}$, $\bw\mapsto \bw^{i+1,i}$ with rate $N^4r^+(w_i)$ and $N^4r^-(w_i)$, respectively, where $r^\pm(w) = e^{-3K\pm Kw}$ and
\beqsn
&\mbf w^{i,i+1} = \mbf w -{\bf d}^{(4)},\quad\mbf w^{i+1,i}=\bw+{\bf d}^{(4)},\\
&{\bf d}^{(4)}=\mbf e^{i-2} - 4\mbf e^{i-1} + 6\mbf e^{i} - 4\mbf e^{i+1}+\mbf e^{i+2}.
\eeqsn
\subsection{Role of Height, Slope, and Third Order FD Process}\label{subsec:role}
Now that we have introduced the three processes $\hN$, $\zN$, and $\wN$, let us explain their roles in this paper. The original $\hN(t)$ is the physically meaningful process, and our main goal is to derive the PDE governing its hydrodynamic limit $h$. However, it will be more convenient to first study the hydrodynamic limit $w$ of $\wN$, and to deduce the PDE governing $h$ from the PDE governing $w$. As an indication of why this is more convenient, recall the evolution equation~\eqref{evoln-h} for $\E[h_i(t)]$. Using~\eqref{met-rates-general} and the definition of $w_i$, we can now write the current $J^{i,i+1}$ as $
J^{i,i+1}(\bh) = J(w_i),$ where \beq J(w)=(r^+-r^-)(w) = 2e^{-3K}\sinh(Kw).\eeq
Therefore, with the Metropolis rates~\eqref{met-rates-general}, the evolution equation~\eqref{evoln-h} takes the form 
\beq\label{evoln-h-II}
\partial_t\E[h_i(t)] = -N^4\E\l[J(w_{i}(t))-J(w_{i-1}(t))\r].
\eeq
Thus, the evolution of $\E[h_i]$ depends nonlinearly on the FDs $w_i$, $w_{i-1}$. In hydrodynamic limit derivations, such dependence on finite differences is inconvenient. But if we take the third order FD of both sides of~\eqref{evoln-h-II}, we get 
\beq\label{evoln-w}
\partial_t\E[w_i(t)] = -N^4\E\left[J(w_{i-2})-4J(w_{i-1}) +6J(w_i)- 4J(w_{i+1}) + J(w_{i+2})\right].\eeq
We see that the evolution of $\E[w_i]$ can be written in terms of $w_j$'s alone. 

Now, let us address the role of $\zN$.  Roughly speaking, the PDE governing the limit $w$ comes from replacing $\E[J(w_i)]$ in the righthand side of~\eqref{evoln-w} by $\hatJ(w(t, i/N))$ for some function $\hat J$. 
Showing such a replacement is possible and determining $\hat J$ will require us to have some knowledge of $\Law(\wN(t))$. Since the distribution $\Law(\zN(t))$ determines the distribution $\Law(\wN(t))$, we could study the former to understand the latter. As a helpful starting point, it turns out that $\zN$ has the special property that it is reversible with respect to the standard Gibbs measure $$\Phi_N(\mbf z)\propto \e(-KH(\mbf z)) = \e(-K\sum_{i=1}^N z_i^2),\quad \bz\in\Z^N.$$ This is a result of detailed balance, i.e. $$r^{i,i+1}(\mbf z)\Phi_N(\mbf z) = r^{i+1,i}(\mbf z^{i,i+1})\Phi_N(\mbf z^{i,i+1})$$ for all $i$, which is easy to see using the original formulation of the rates~\eqref{rates-H}. 
\begin{remark}\label{rk:metrop}
Any rates of the form $r^{i,j}(\bz)=\psi(\Delta H)$ which satisfy $\psi(-\Delta H) = \psi(\Delta H)e^{K\Delta H}$ are in detailed balance with the Gibbs measure $\Phi_N$. Another example of rates in this family is $\psi(\Delta H) = e^{-K\Delta H}\wedge1$, which is the acceptance probability in a Metropolis-Hastings scheme to sample from $\Phi_N$. This is where the name ``Metropolis'' comes from. 
\end{remark}
Reversibility of $\zN$ with respect to $\Phi_N$ suggests that for $N$ large, $\Law(\zN(t))$ is a \emph{local} Gibbs product measure of the form
\beq\label{gibbs-measure}
\mathbb P(\zN(t)=\mbf z) \approx \rhogibbsNd(\mbf z)\propto \e\l(-K\sum_{i=1}^Nz_i^2 + 2K\sum_{i=1}^N\lambda_iz_i\r)
\eeq
for some $\lambda_i=\lambda_i(t)$. There are deeper and more technical reasons why the local Gibbs measure typically arises, which we will not get into here. See e.g.~\cite{varadhanI,kipnisbook} for a rigorous probabilistic treatment of this topic and~\cite{spohnbook} for a more physical treatment. 


The paper~\cite{gao2020} assumed that $\Law(\zN(t))$ is a local Gibbs distribution to carry out the aforementioned replacement and to determine the function $\hat J$. However, the authors gave preliminary evidence that interestingly enough, the local Gibbs distribution is not accurate for all $K$. And indeed, we will show definitively in Section~\ref{subsec:gibbs} that $\Law(\zN(t))$ \emph{cannot} be approximated by a local Gibbs distribution as $N\to\infty$. Why the local Gibbs distribution is not the correct form of $\Law(\zN(t))$ is a very interesting question worthy of further investigation, but we do not pursue the question here. We will see that despite being incorrect, the local Gibbs approximation~\eqref{gibbs-measure} to $\Law(\zN(t))$ leads to a crude but numerically useful approximation to the true function $\hat J$, the computation of which is the main goal of this paper. Beyond this approximation, however, the $\zN$ process will play no role in our PDE derivation. 

\subsection{A Close Cousin: Arrhenius Rate Dynamics}\label{subsec:arr} Throughout the paper, it will often be helpful to compare the Metropolis rate process and its hydrodynamic limit to the Arrhenius rate process and its limit, which were studied in~\cite{kat21}. For the sake of a self-contained paper, let us review the key features of the Arrhenius process. We will let $\hN^\arr$ denote the height process, $\zN^\arr$ denote the first order FD (slope) process, and $\wN^\arr$ denote the \emph{second} order FD process. The Arrhenius rates $r_N^{i,j}$ can also be written $r_N^{i,j}(\bh)=N^4r^{i,j}(\bh)$, where $r^{i,j}(\bh)=r^{i,j}(\bz)$. They are symmetric with respect to jumping left and right, with $r^{i,i\pm1}(\bz)=r(z_i-z_{i-1})$ for  $r(w)=e^{-2K-2Kw}$.  For the physical interpretation of these rates, see~\cite{kat21} and the references therein. Like the Metropolis rates, the Arrhenius rates are reversible with respect to the Gibbs measure $\Phi_N(\bz)\propto\e(-KH(\bz))$. But unlike $\Law(\zN(t))$ for the Metropolis process, the distribution $\Law(\zN^\arr(t))$ \emph{does} converge to a local Gibbs measure as $N\to\infty$. This is the key difference between these two otherwise very similar processes. Another similarity is that the evolution of $\E[w_i^{\arr}(t)]$ takes the exact same form as the evolution~\eqref{evoln-w} of $\E[w_i(t)]$, except that the function $J$ is replaced with the function $r$. 

\section{Hydrodynamic Limit in the Rough Scaling Regime}\label{sec:hydro}
In this section, we define the hydrodynamic limit of a Markov jump process $\vN(t)\in\Z^N$ under a given scaling regime. We then specify the rough scaling regime for the Metropolis $\hN$ process, and motivate recasting the limit of $\hN$ in terms of the limit of $\wN$.  

Let $\vN(t)\in\Z^N$, $N=1,2,\dots$ be a sequence of Markov jump processes on the periodic lattice $\{1,2,\dots,N\}$ with transitions $\bv\mapsto\bv^{i,j}$ occurring at rates $$r_N^{i,j}(\bv)= N^\alpha r^{i,j}(\bv)$$ for some $\alpha$. In order for a hydrodynamic limit to exist, the rates and transition rules should satisfy certain conditions. We will content ourselves with taking $\vN$ to be one of $\hN$ or $\wN$, for which these conditions are satisfied. 

The hydrodynamic limit of $\vN$ arises by rescaling three characteristic scales: time, space, and ``amplitude". The $N^\alpha$ time rescaling has already been incorporated into the transition rates $r_N^{i,j}$. The spatial scaling occurs by identifying the $N$ lattice sites with points on the periodic unit interval (torus), denoted $\unit$. Specifically, we identify $\vN(t) = (v_1(t), \dots, v_N(t))$ with a random measure on the unit interval:
$$\vN(t)\quad\leftrightarrow \quad v_N(t,dx) = \frac1N\sum_{i=1}^Nv_i(t)\delta\l(x-\frac iN\r).$$ Another equivalent way to think of $\vN(t)$ is as a step function, with value $v_i(t)$ in the interval $[i/N, (i+1)/N)$. For the amplitude rescaling, we assume that the $v_i$ grow with $N$, so that to obtain a finite macroscopic limit, the $v_i$ must be scaled down. We will incorporate the amplitude rescaling into the following definition of a hydrodynamic limit:
\begin{definition}[Hydrodynamic Limit,~\cite{kat21}]\label{def:hydro} Suppose $\vN(0)$ is initialized in a random configuration for which there exists $v_0:\unit\to\R$ such that
\begin{align*}\label{init}\frac1N\sum_{i=1}^N\phi\l(\frac iN\r)(N^{-\beta}v_i(0)) \probto\int_\unit\phi(x)v_0(x)dx,\quad N\to\infty.\tag{init}\end{align*} 
We say $\vN$ converges hydrodynamically to $v:(0,T]\times\unit\to\R$ under amplitude scaling $N^{\beta}$ and implied time scaling $N^\alpha$ if for each $t\in(0,T]$, $\phi\in \cts$, we have
\beqs\label{hydro-lim}\frac1N\sum_{i=1}^N\phi\l(\frac iN\r)(N^{-\beta}v_i(t)) \probto\int_\unit\phi(x)v(t,x)dx,\quad N\to\infty.\eeqs \end{definition} Here, the notation $\probto$ denotes convergence in probability. The probability distribution of $\int \phi(x)N^{-\beta}v_N(t,dx)$ is induced by $\Law(\vN(t))=\Law(\vN(0))\e(\LN t)$, where $\LN$ is the generator of $\vN$. 

Note that the lefthand side of~\eqref{hydro-lim} equals $\int \phi(x)N^{-\beta}v_N(t,dx)$, so that~\eqref{hydro-lim} expresses that the random measure $N^{-\beta}v_N(t,dx)$ converges to the measure $v(t,x)dx$. 
\begin{definition}[Rough Scaling Regime]
Let $\hN$ be governed by the Metropolis dynamics specified in Section~\ref{subsec:micro}. We say $\hN$ converges to $h:[0,T]\times\unit\to\R$ in the \textbf{\emph{rough scaling regime}} if $\hN$ converges hydrodynamically to $h$ under amplitude scaling $N^{3}$ and implied time scaling $N^4$.\end{definition}
We will explain the choice $\alpha=4$ and $\beta=3$ below. 
As an example of a distribution on $\hN(0)$ satisfying~\eqref{init} with $\beta=3$, consider a product measure with marginals $$h_i\;\sim\;\lfloor N^3 h_0(i/N)\rfloor + \xi_i,$$ where $\lfloor q\rfloor$ denotes the integer part of $q$ and $\xi_i$, $i=1,\dots, N$ are i.i.d. integer-valued random variables with bounded support. In fact,~\eqref{init} is satisfied as long as $|\xi_i|<B_N$ where $B_N = o(N^3)$. To summarize the rough scaling regime in simple terms, start with an $O(1)$ height profile $h_i(0) = h_0(i/N)$. Then, multiply it by $N^3$ to get $\hN(0)$, and evolve it forward according to the time-rescaled Metropolis rate dynamics. To get a hydrodynamic limit, divide $\hN(t)$ by $N^{3}$ and take $N\to\infty$. It may seem like multiplying and dividing by $N^3$ should have no effect. But this is not so, because scaling has a nonlinear effect on the dynamics, so that different choices of $\beta$ for the amplitude scaling lead to different hydrodynamic limits. A straightforward way to see this is to note that the rates $r^{i,i+1}$ and $r^{i+1,i}$ are exponential in $w_i$. Scaling $w_i$ by a constant multiple will affect the rates nonlinearly, and thus have a nonlinear effect on the evolution of $h_i$. To understand the effect of different $\beta$ in more detail, suppose for the moment that we have $\E[h_{i}(t)]\approx N^{\beta}h(t,i/N) $ for $N\gg1$. (This of course does not follow from hydrodynamic convergence.) Using~\eqref{evoln-h-II}, we should then have
\beq\label{ab-scaling}\partial_th(t,i/N) \approx -N^{4-\beta}\left(\E J\left(w_{i}\right) - \E J\left(w_{i-1}\right)\right).\eeq Now, if $h_i$ has order $N^\beta$ then $w_i$ should have order $N^{\beta-3}$, since it is the third order FD of $h_i$ (this statement is purely formal; see below). If $\beta<3$, then we expect that in the $N\to\infty$ limit, the nonlinear function $J$ will become linearized around $0$. Taking $\beta=3$ as in the rough scaling limit,  the nonlinear function $J$ is in some sense ``preserved" as $N\to\infty$, leading to a very different PDE. The reason for the $N^{4}$ time scaling is that it ensures that the total power of $N$ is $4-3=1$ in~\eqref{ab-scaling}, which balances the order of the finite difference $J(w_{i}) - J(w_{i-1})$. In sum, the rough scaling regime is the unique choice of $\alpha,\beta$ which leads to a nontrivial and non-exploding limit $h$ governed by a PDE which ``preserves" the nonlinear function $J$ (we use quotation marks because the PDE will involve not $J$ but a related $\hat J$, also with exponential nonlinearity). Note that this choice is tailored to the Metropolis dynamics. For the Arrhenius dynamics, for example, $\alpha=4,\beta=2$ gives the PDE with exponential nonlinearity.  

Of course, if $h_i=O(N^3)$ then in general we cannot infer $w_i=O(1)$, since taking finite differences is unstable. This motivates us to take $\wN$ as our ``original" process and study its hydrodynamic limit under $N^0$ amplitude scaling. We then expect to obtain the hydrodynamic limit of $\hN$ under $N^3$ amplitude scaling by doing three cumulative sum operations. We will carry out this program formally in the next section. In particular, we will see that the function $\hat J$ of the macroscopic current $\hat J(h_{xxx})$ in the $h$ PDE is intrinsically linked to the $\wN$ process. 

\section{PDE for $h$ via Third Order Finite Differences}\label{sec:theory}
This section explains our approach to deriving the PDE governing the hydrodynamic limit of $\hN$ in the rough scaling regime, via the hydrodynamic limit of $\wN$. We begin the section with an overview of this approach. First, we will show that
\beq\label{w-to-h-logic}
\wN\stackrel{\substack{\text{ptwise}\\\text{meso}}}{\to} w\quad\Rightarrow\quad\wN\stackrel{\text{hydro}}{\to} w\quad\Rightarrow\quad \hN\stackrel{\text{hydro}}{\to} h.
\eeq
The rightmost limit denotes hydrodynamic convergence of $\hN$ in the rough scaling regime: amplitude scaling $N^3$, time scaling $N^4$. We will show this follows from the middle limit: hydrodynamic convergence of $\wN$ under amplitude scaling $N^0$ and time scaling $N^4$. The limiting function $h$ will be uniquely determined from the function $w$, the initial macroscopic condition $h_0$, and the periodic boundary. The leftmost limit denotes ``pointwise mesoscopic" convergence of $\wN$, which is nonstandard but physically intuitive, and was used in our study of rough local equilibria in~\cite{kat21}. We will show that pointwise mesoscopic convergence implies hydrodynamic convergence.

Next, consider the following key approximation:
\beq\label{Ef-intro} \frac{1}{2N\epsilon}\sum_{i\in\idxsetx}\E J(w_i(t)) \approx \hat J\l(\frac{1}{2N\epsilon}\sum_{i\in\idxsetx} \E w_i(t)\r),\quad N\gg1, \epsilon\ll1.\eeq
The existence of $\hat J$ such that~\eqref{Ef-intro} holds is a property of locally equilibrated processes, as we will explain in Section~\ref{subsec:LE}. It is important to note that~\eqref{Ef-intro} is \emph{not} an assumption. In rigorous hydrodynamic limit arguments, proving the so-called ``Replacement Lemma", which is analogous to~\eqref{Ef-intro}, is typically the central and most difficult part (for more on this, see the discussion and references in~\cite{kat21}). We will show numerically that~\eqref{Ef-intro} is satisfied for the Metropolis process. The equation is the key ingredient to derive the PDE since, as we will show in Claim~\ref{claim:PDE},
\beqs\label{EplusEfimplyPDE}
\bigg(\wN\stackrel{\substack{\text{ptwise}\\\text{meso}}}{\to} w\bigg)\quad &+\quad \bigg(\exists\,\hat J\text{ s.t. \eqref{Ef-intro} holds }\forall\;t,x\bigg)\\
&\Rightarrow\quad w\text{ solves }\partial_tw = -\partial_{xxxx}\hatJ(w)\text{ weakly.}
\eeqs
From here, we will be able to conclude that $h$, the hydrodynamic limit of $\hN$ in the rough scaling regime, is the weak solution to the PDE $$\partial_th= -\partial_{x}\hatJ(h_{xxx}).$$

Thus, if we can verify the conditions in~\eqref{EplusEfimplyPDE} --- that $\wN$ converges pointwise mesoscopically to $w$, and a function $\hat J$ exists such that~\eqref{Ef-intro} holds for all $t,x$ --- then the chain of logic just described will lead us to the PDE for $h$, our original goal. More specifically, this logic establishes the \emph{form} of the PDE, but it remains to compute $\hat J$. Doing so numerically will be the focus of Section~\ref{sec:num}. 

The assertions~\eqref{w-to-h-logic} and~\eqref{EplusEfimplyPDE} will be formalized in Section~\ref{subsec:theory} and proved rigorously in Appendix~\ref{app:claims}. The rigorous proofs rely on the following supplementary boundedness assumptions:
\begin{align*}\label{bd1}\sup_{N}\sup_{i=1,\dots, N}\E|w_i(t)| < \infty\quad\forall t\geq0,\tag{w-bd}\end{align*}
\begin{align*}\label{bd2}
\sup_{N}\sup_{s\in(0,t]}\sup_{i=1,\dots,N}|\E J(w_i(s))|<\infty\quad\forall t\geq0.\tag{J-bd}\end{align*}
 These assumptions are extremely strong (most likely unnecessarily so), but our primary aim in presenting the proofs is to put our numerical method on firm footing. We will numerically check both the boundedness assumptions and the two conditions of~\eqref{EplusEfimplyPDE} in Section~\ref{subsec:theory-numeric}. 

Later in Section~\ref{sec:num}, we will verify numerically that the end goal has been achieved: that $\hN$ does in fact converge to the solution of the PDE we obtain. Given this, the numerical and theoretical verifications of this section may seem unnecessary. Their purpose is to confirm that we have obtained the correct PDE \emph{for the correct reason}. This is important because hydrodynamic limit derivations can be delicate. For example, in~\cite{mw-krug} the authors used heuristic arguments to derive the PDE limit of the Arrhenius dynamics in the rough scaling regime. They confirmed numerically that the PDE they obtained is correct. However, we show numerically in~\cite{kat21} that some of the assumptions in~\cite{mw-krug} were incorrect, which  obscured the true reason the PDE takes the form it does (see the discussion in Section 6.4 of~\cite{kat21}).

\subsection{From $w$ to $h$: Limit and PDE}\label{subsec:theory}

We start by showing that the hydrodynamic limit of $\hN$ is determined from the hydrodynamic limit of $\wN$. 
The following claim corresponds to the second (righthand) implication in~\eqref{w-to-h-logic}. 
\begin{claim}\label{claim:htow}
Let $\hN(t)$, $N=1,2,\dots$ be a sequence of Metropolis rate height processes, and $\wN(t)$ be the corresponding third order FD processes. Suppose $$N^{-1}\sum_{i=0}h_i(0)\probto M,\quad N\to\infty,$$ for some deterministic $M$ and that~\eqref{init} is satisfied for $\wN$ with $\beta=0$ and some $v_0=w_0$. If $\wN$ converges hydrodynamically to a function $w$ under amplitude scaling $N^0$, and if~\eqref{bd1} holds, then $\hN$ converges hydrodynamically under amplitude scaling $N^3$. The limit $h$ is the unique periodic function such that $\int_0^1h(t,x)dx = M$, $h_{xxx}=w$, and such that $h_x$, $h_{xx}$ are also periodic. 
\end{claim}
For the proof of the claim see Appendix~\ref{app:claims}. 
The claim implies in particular that~\eqref{init} is satisfied for $\hN$ with $\beta=3$, where $h_0$ is uniquely determined from the function $w_0$, the constant $M$, and the periodic boundary.

We will now recall from~\cite{kat21} the notion of pointwise mesoscopic convergence, which will be very convenient to study from a numerical perspective. We begin with some notation. For a vector $\bf w$ and a function $f:\R\to\R$, define
\beqs\label{wfbar}
\varbar w = \frac{1}{2N\epsilon}\sum_{i\in\idxsetx} w_i,\qquad \fxnbar f w= \frac{1}{2N\epsilon}\sum_{i\in\idxsetx}f(w_i).\eeqs 
Here, $i\in N(x\pm\epsilon)$ denotes $i\in 1,2,\dots, N$ such that $|(i/N-x)\bmod1|\leq\epsilon$. 
\begin{remark}\label{rk:w-meas} Let $(\phi\ast\mu)(x)=\int_\unit\phi(x-y)\mu(dy)$ for a function $\phi:\unit\to\R$ and a signed measure $\mu$ on $\unit$. Using the interpretation of $\wN(t)=\wN(t,dx)$ as a signed, random measure, note that $\varbar w(t)$ can be written in the following way: 
\beq
\varbar w(t) = (\phi_\epsilon\ast\wN(t))(x),\quad\text{where}\quad \phi_\epsilon(x)=\mathbbm{1}_{(-\epsilon,\epsilon)}(x)/2\epsilon.
\eeq
\end{remark}
\begin{definition}[Pointwise Mesoscopic Convergence] We say $\wN$ converges pointwise mesoscopically if there exists a continuous function $w$ such that
\beq\label{convergence}\limNeps\E\bigg|\varbar w(t) - w(t,x)\bigg|^2=0,\quad\forall t\geq0,x\in\unit.\eeq 
\end{definition}
The reason we think of~\eqref{convergence} as a ``pointwise" convergence is that it holds for each $x$, and the \emph{limit} is the pointwise quantity $w(t,x)$. However, $\varbar w(t)$ is not itself a ``pointwise" quantity, but rather an average over the \emph{mesoscopic}, or \emph{local} set $\idxsetx$: for each fixed $\epsilon\ll1$, this set contains a number of microscopic lattice sites that grows to infinity with $N$. At the same time, it corresponds to the small macroscopic interval $(x-\epsilon,x+\epsilon)$.
A practical reason to consider pointwise mesoscopic convergence is that it connects with our numerical method to compute $\hat J$, which uses the local quantities $\varbar w$ and $\fxnbar J w$. As such, it will be more straightforward to prove that $w$ solves $w_t=-\partial_{xxxx}\hat J(w)$ if we know that $w$ is the pointwise mesoscopic, rather than hydrodynamic, limit of $\wN$.  Moreover, as noted in~\cite{kat21}, it is convenient that convergence in $L^2$ (with respect to randomness) can be separated into convergence of expectations and vanishing variance. Namely,~\eqref{convergence} is equivalent to 

\begin{align*}
\label{V}&\lim_{\epsilon\to0}\lim_{N\to\infty}\mathrm{Var}\left(\varbar w(t)\right)= 0\quad\forall x\in\unit,\tag{V}\\ 
\label{E}&\text{There exists a continuous $w(t,\cdot):\unit\to\R$ such that }\tag{E}\\
&\lim_{\epsilon\to0}\lim_{N\to\infty}\E\varbar w(t)= w(t,x), \quad \forall x\in\unit,\\ \\
\end{align*}
for all $t\geq0$.
We now formalize the first (lefthand) implication in~\eqref{w-to-h-logic}.
\begin{claim}\label{claim:meso}
Assume that
$\wN$ converges pointwise mesoscopically to $w$, and that~\eqref{bd1} holds. Then
\beq\int \phi(x)w_N(t,dx) \probto\int_\unit\phi(x)w(t,x)dx,\quad N\to\infty\eeq for all $\phi\in\cts$ and $t\geq0$. In other words, $\wN$ converges hydrodynamically to $w$ under amplitude scaling $N^0$. 
\end{claim}
The proof is given in Appendix~\ref{app:claims}. Now that we have discussed convergence of $\wN$, we turn to the problem of deriving the PDE governing its limit $w$. To do so, we will exploit the following crucial property of any process with a ``rough" local equilibrium state (defined in Section~\ref{subsec:LE}):
\begin{align*}\label{Ef}&\text{For all ``suitable" }f,\text{ there exists continuous }\hat f\text{ such that  }\tag{Ef}\\ &\E\bar {f}({\bf w}_{\idxsetx }(t))\stackrel{N,\epsilon}{\approx}\hat f\left(\E\varbar w(t)\right), \quad \forall x\in\unit.\end{align*} We use the notation $A\stackrel{N,\epsilon}{\approx}B$ to mean that $|A-B|$ converges to zero as $N\to\infty$ and then $\epsilon\to0$. Note that if $\eE$ and $\eEf$ both hold then 
\beq\label{eEfeE}
\limNeps\E\bar {f}({\bf w}_{\idxsetx }(t)) = \hat f(w(t,x)).
\eeq
The ``suitable" functions $f$ are discussed in Section~\ref{subsec:LE}. For the Metropolis process, we confirm $\eEf$ for both $f(w)=w^2$ and $f(w)=J(w)$ in Section~\ref{subsec:theory-numeric}. However, to derive the Metropolis PDE we will only use that $\eEf$ holds for $f=J$. The following claim formalizes the assertion~\eqref{EplusEfimplyPDE} in the introduction to this section. 
\begin{claim}\label{claim:PDE}
Suppose $\wN$ converges pointwise mesoscopically to $w$, and that $\eEf$ holds with $f=J$ for all $t>0$. Also, assume the bounds~\eqref{bd1} and~\eqref{bd2}. Then $w$ is a weak solution to
\beq\label{w-PDE}
\begin{cases}
\partial_tw(t,x) &= -\partial_{xxxx}\hatJ(w),\quad t>0,x\in \unit,\\
w(0,x)&=w_0(x),\quad x\in\unit\end{cases}\eeq
in the sense that \beq\label{weak-def}\int_0^1\psi(x)\l[w(t,x)-w_0(x)\r]dx = -\int_0^t\int_0^1\psi^{(4)}(x)\hatJ(w(s,x))dxds,\eeq $\forall t>0,\psi\in C^4(\unit)$. \end{claim}
\begin{proof} First, we substitute $w(t,x)-w_0(x)$ on the lefthand side of~\eqref{weak-def} by the limit of $\E\l[\varbar w(t)-\varbar w(0)\r]$. Thanks to~\eqref{bd1}, we can pull the limit outside of the integral. Thus, the lefthand side is the limit of $\int_\unit\psi(x)\E\l[\varbar w(t)-\varbar w(0)\r]dx$ as $\Nepslim$. Now, as in Remark~\ref{rk:w-meas}, note that we can write $$\E\varbar w(t)=(\phi_\epsilon\ast\E[\wN(t)])(x),$$ where $\E[\wN(t)]$ is the signed measure which assigns weight $\E[w_i(t)]$ to $x=i/N$.  We can then use the identity $\int_\unit \psi(x)(\phi\ast\mu)(x)dx = \int_\unit (\psi\ast\phi)(x)\mu(dx)$ which holds for even functions $\phi$. Thus, we get that
\beqs\label{PDE-derivation-i}
\int_\unit\psi(x)\E&\l[\varbar w(t)-\varbar w(0)\r]dx \\&= \int_\unit(\psi\ast\phi_\epsilon)(x)\E\l[ w_N(t,dx)-w_N(0,dx)\r]\\&= \frac1N\sum_{i=1}^N(\psi\ast\phi_\epsilon)\l(\frac iN\r)\E[w_i(t) - w_i(0)]\\
&= -\frac1N\sum_{i=1}^N(\psi\ast\phi_\epsilon)\l(\frac iN\r) \int_0^t\E[N^4D_N^4J(w_i(s))]ds,
\eeqs
where $D_N^4J(w_i) = J(w_{i-2})-4J(w_{i-1}) +6J(w_i)- 4J(w_{i+1}) + J(w_{i+2})$. We can now move $N^4D_N^4$ onto $(\psi\ast\phi_\epsilon)(i/N)$. For $N\gg1$, the result is approximately $(\psi^{(4)}\ast\phi_\epsilon)(i/N)$. Then we move $\phi_\epsilon$ back onto $\E[J(w_i(s))]$, to arrive at
$$\int_\unit\psi(x)\E\l[\varbar w(t)-\varbar w(0)\r]dx \approx -\int_\unit\psi^{(4)}(x)\int_0^T\E\bar J(\bw_\idxsetx(s))dsdx.$$ We now apply~\eqref{eEfeE} with $f=J$, and the bound~\eqref{bd2}, to conclude by Dominated Convergence. The details of the proof are filled in in Appendix~\ref{app:claims}.
\end{proof}

Finally, we return to our original goal to derive the PDE governing the rough scaling limit of $\hN$.
\begin{corollary}
Let $\hN(t)$ be a Metropolis rate process such that $N^{-1}\sum_{i=1}^Nh_i(0)$ converges to some constant $M$ in probability, and assume the conditions of Claim~\ref{claim:PDE}. Then $\hN$ has a hydrodynamic limit $h$ in the rough scaling regime which is three times continuously differentiable in $x$, and which is the weak solution to 
\beq\label{h-PDE}
\begin{cases}
\partial_th(t,x) &= -\partial_{x}\hatJ(h_{xxx}),\quad t>0,x\in \unit,\\
h(0,x)&=h_0(x),\quad x\in\unit\end{cases}\eeq
in the sense that
\beq\label{weak-def-ii}\int_0^1\psi(x)\l[h(t,x)-h_0(x)\r]dx = \int_0^t\int_0^1\psi'(x)\hatJ(h_{xxx}(s,x))dxds\eeq 
for all $t>0$ and $\psi\in C^1(\unit)$. 
\end{corollary}
\begin{proof}
By Claim~\ref{claim:meso}, $\wN$ converges hydrodynamically to $w$, and by Claim~\ref{claim:htow}, $\hN$ then converges hydrodynamically to the unique periodic $h$ such that $\int h(t,x)dx=M$ for all $t$ and $h_{xxx}=w$. Also, Claim~\ref{claim:PDE} gives that $w$ is the weak solution to $w_t=-\partial_{xxxx}\hatJ(w)$. Now, note that~\eqref{weak-def-ii} is clearly satisfied for $\psi\equiv1$, so it suffices to show~\eqref{weak-def-ii} for all $\psi\in C^1(\unit)$ which integrate to zero. For such $\psi$, there exists a function $\phi\in C^4(\unit)$ such that $\phi^{(k)}$, $k=0,1,2$ are all periodic and such that $\phi_{xxx}=\psi$. We substitute $\phi_{xxx}=\psi$ into the lefthand side of~\eqref{weak-def-ii}, integrate by parts, and use the fact that $w=h_{xxx}$ satisfies~\eqref{weak-def}.
 \end{proof}

So far, we have only established the \emph{form} of the PDEs governing $w$ and $h$. We must now actually compute the function $\hat J$. Note that according to $\eEf$ the points $(\E\varbar{w}(t), \E \bar J\l({\bf w}_\idxsetx(t)\r))$ should lie on the curve $\{(\omega,\hatJ(\omega))\mid\omega\in\R\}$. Thus, we can compute $\hat J$ numerically simply by interpolating these points! This is the essence of our numerical method, described in full in Section~\ref{sec:num}.

But is it possible to compute $\hat J$ analytically instead of resorting to numerics? To address this question, we need to explain why we expect a function $\hat J$ satisfying $\eEf$ to exist in the first place. This has to do with the form of the \emph{local equilibrium} (LE) distribution of $\wN$, discussed in Section~\ref{sec:LE}.  


\subsection{Numerical Verification of Claim Assumptions}\label{subsec:theory-numeric}
Let us now check numerically the assumptions of Claims~\ref{claim:meso} and~\ref{claim:PDE}. Namely, we need to check $\eE$, $\eV$,~\eqref{bd1},and~\eqref{bd2}. Each of these conditions can be written in terms of expectations of the form $\E[f(\wN(t))]$. For details on how we estimate such expectations numerically, see Section~\ref{subsec:setup}.
\begin{figure}
\begin{subfigure}[b]{0.48\textwidth}
\includegraphics[width=\textwidth]{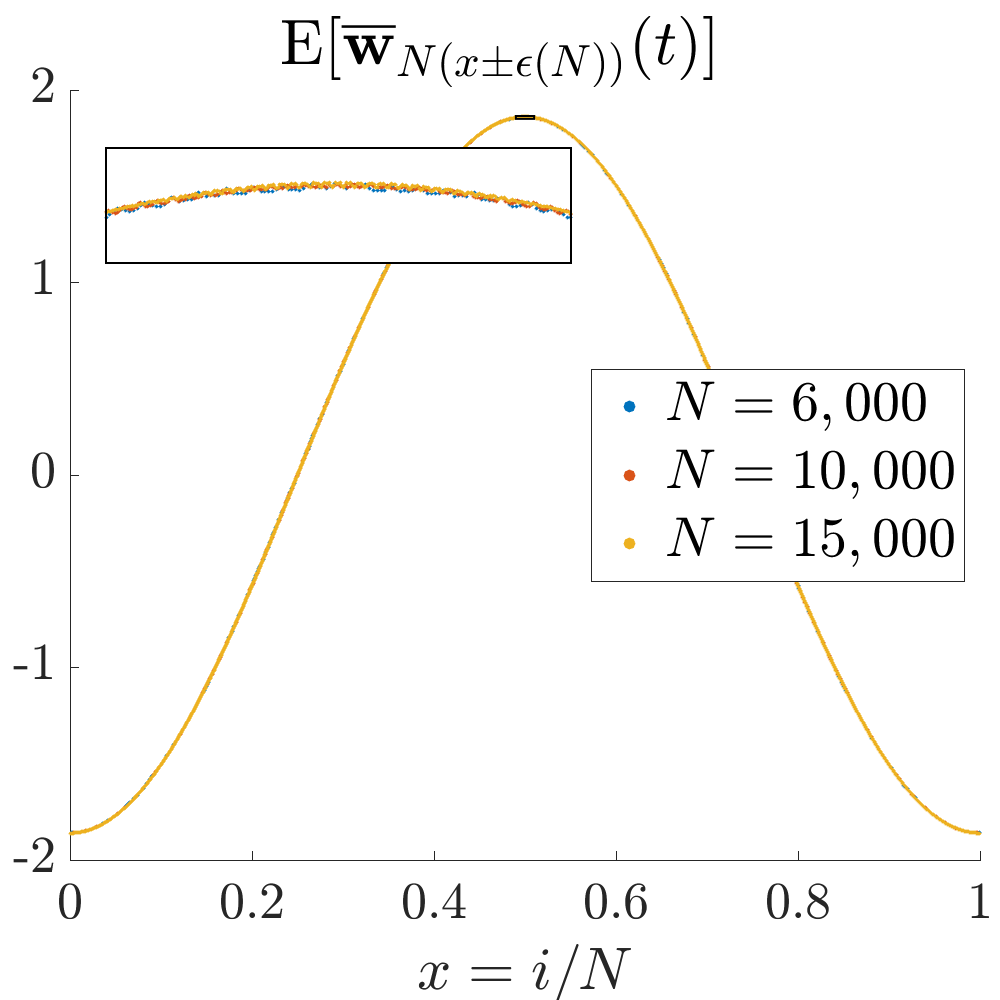}
\caption{}
\label{fig:E}
\end{subfigure}
\begin{subfigure}[b]{0.48\textwidth}
\includegraphics[width=\textwidth]{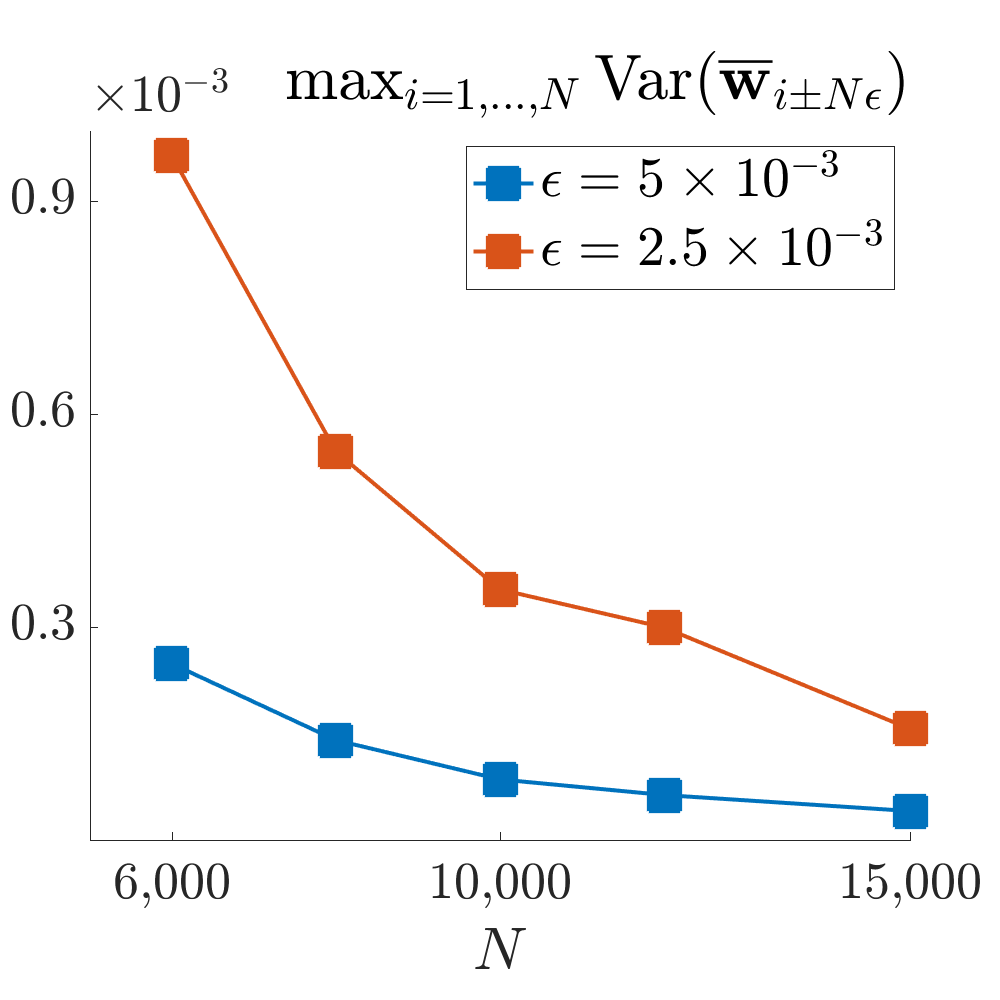}
\caption{}
\label{fig:V}
\end{subfigure}
\caption{In (a), we confirm $\eE$. The figure shows that $\E\bar{\bf w}_{N(x\pm\epsilon(N))}(t)$ converges as $N\to\infty$ to some $w$. We choose $\epsilon=\epsilon(N)$ as a proxy for taking a double limit; see the text for details. In (b), we confirm $\eV$. As expected, the variance of the window average goes to zero as $N\to\infty$ for each fixed $\epsilon$.}
\vspace{-0.7cm}
\end{figure}

Figure~\ref{fig:E} and~\ref{fig:V} confirm that the two limits $\eE$ and $\eV$ hold. In both figures, $\wN(t)$ is computed as the third order finite difference of $\hN(t)$, generated from an initial condition $\hN(0)$ satisfying~\eqref{init}, $\beta=3$, with $h_0(x) = 0.0075\sin(2\pi x)$and such that $\wN(0)$ satisfies~\eqref{init}, $\beta=0$, with $w_0=h_0'''$.

The double limit $\eE$ as $N\to\infty$ and $\epsilon\to0$ is delicate. This is because, as we show later in Figure~\ref{fig:gazon}, the profile $i/N\mapsto \E w_i$ varies roughly, but we want to show their sliding window averages converge to a smooth limit. We cannot take $N\to\infty$ numerically, and for every finite $N$, if $\epsilon$ is small enough (e.g. smaller than $1/N$), the window average $\E\varbar w(t)$ will revert back to being roughly varying. We therefore cannot take $\epsilon$ too small. We circumvent this problem with the following heuristic. For each $N$, we choose a ``good'' $\epsilon(N)$: for $\epsilon>\epsilon(N)$, $\E\varbar w(t)$ is smooth but biased, whereas for $\epsilon<\epsilon(N)$, it is unbiased but rough. We then check that $\E\overline{\bf w}_{N(x\pm\epsilon(N))}$ is converging as $N\to\infty$, as shown in Figure~\ref{fig:E}. Figure~\ref{fig:V} shows that the variance of the window average decays as $N\to\infty$ for each fixed $\epsilon$, i.e. as the window size increases. This suggests that pairs $w_i, w_j$, $i\neq j$ are uncorrelated or have low correlation.  

Figure~\ref{fig:Ef} shows that $\eEf$ is satisfied for $f=J$ and $f(w)=w^2$, using the same sinusoidal initial condition. \begin{figure}
\centering
\includegraphics[width=\textwidth]{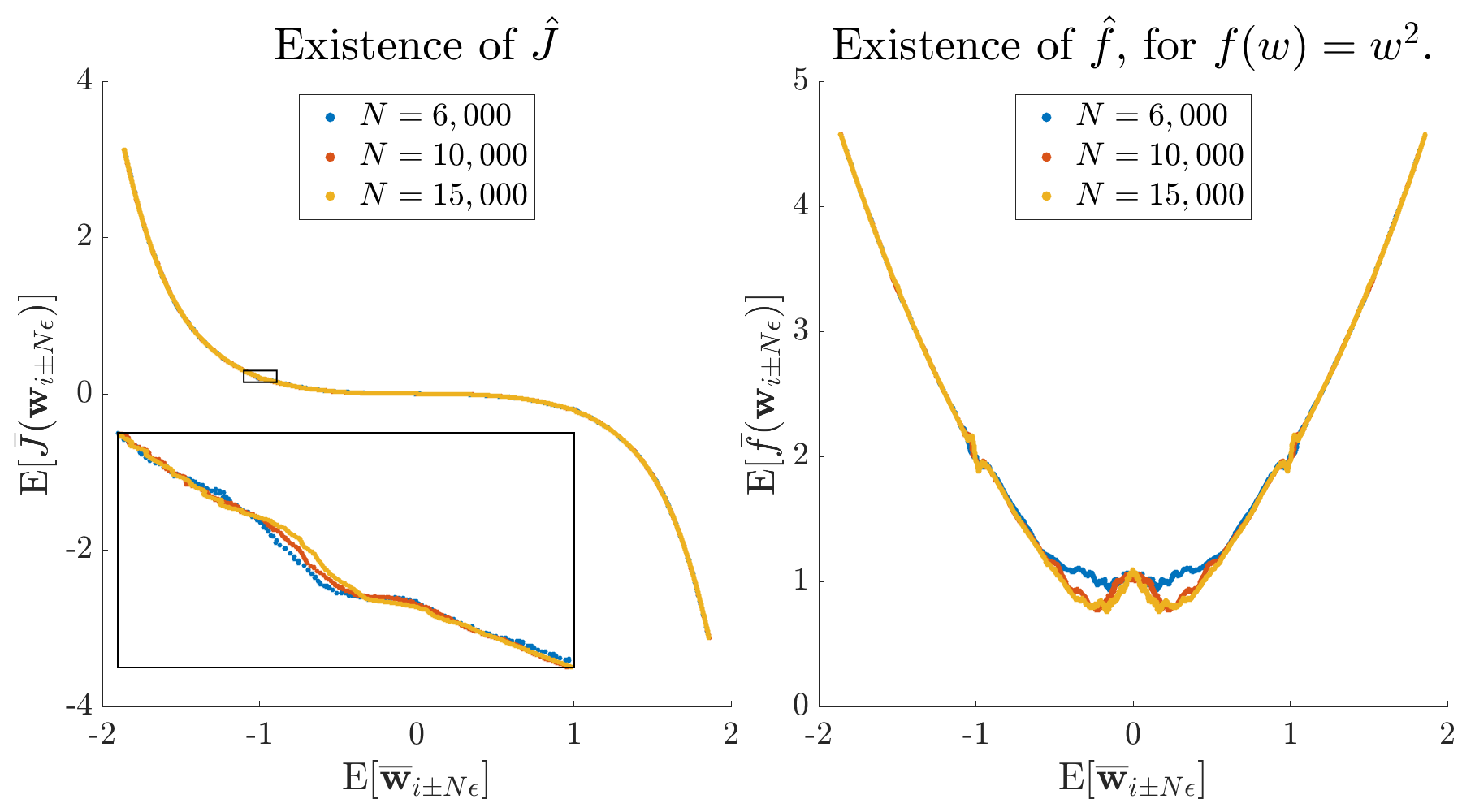}
\caption{Here we confirm $\eEf$ for the two functions $f(w)=J(w) = 2e^{-3K}\sinh(Kw)$ (left) and $f(w)=w^2$ (right). The existence of $\hat J$ (i.e. $\hat f$ for $f=J$) is the key reason a macroscopic dynamics emerges in the hydrodynamic limit. Estimation of $\hat J$ will enable us to determine the PDE numerically.}
\label{fig:Ef}
\vspace{-0.7cm}
\end{figure}
In the figure, we plot the points $(\E\overline{\bw}_{i\pm N\epsilon}(t), \E\bar f(\bw_{i\pm N\epsilon}(t)))$, $i=1,\dots, N$, and confirm that they lie on a fixed curve in the $N,\epsilon$ limit. We take the double limit $\Nepslim$ using the same heuristic as with $\eE$: for each $N$, we choose a ``good" $\epsilon(N)$ for the length of the averaging interval. 

We now turn to the boundedness conditions~\eqref{bd1} and~\eqref{bd2}. The top middle panel of Figure~\ref{fig:gazon} depicts $\E[w_i(t)^2]$, $i=1,\dots, N$ at three points in time at $N=400$. This is evidence for the fact that $\E|w_i(t)|$ remains bounded over time and over $i=1,\dots, N$, since $\E|w_i|\leq \sqrt{\E[w_i^2]}$ and we see that $\max_i\E[w_i(t)^2]$ is decreasing in time. Meanwhile, the bottom panel shows that $\max_{i=1,\dots,N}\E|w_i(t)|$ remains bounded as $N$ increases. Similarly, the top right panel of Figure~\ref{fig:gazon} shows $\E[J(w_i(t))]$ remains bounded over time and over $i=1,\dots, N$, while the bottom right panel shows it remains bounded as $N$ grows. 
\section{Local Equilibrium, but no Local Gibbs} \label{sec:LE} 
Let us return to the questions posed at the end of Section~\ref{subsec:theory}: can we compute $\hat J$ analytically, and why should we expect $\hat J$ satisfying $\E\bar J({\bw}_{\idxsetx})\approx \hatJ(\E\ol{\bw}_\idxsetx)$ to exist at all? To address these questions, we first review the key ideas in~\cite{kat21} on ``smooth" and ``rough" local equilibrium (LE) states. We then show that $\Law(\zN)$ is not a local Gibbs measure.
\subsection{Local Equilibrium}\label{subsec:LE}
This section reviews ideas from~\cite{kat21} and is primarily for the reader's convenience. Informally, a Markov jump process $\vN$ has an LE state if there is an $M$-parameter family of distributions (where $M$ is fixed as $N\to\infty$) such that for each $t$ and $x$, the PDE-relevant information contained in the joint law of $\{v_i(t)\}_{i\in\idxsetx}$ is fully determined by a single measure in this family via some parameters $\lambda_1(t,x),\dots, \lambda_M(t,x)$ specifying this measure. What we mean by ``PDE-relevant" will become clear at the end of the section. 

Here we will only discuss LE states which can be described by a $M=1$ parameter, \emph{mean-parameterized} family $\{\mu[\lambda]\mid\lambda\in\R\}$. Here, each $\mu[\lambda]$ is a probability mass function (pmf) on $\Z$, and ``mean-parameterized" means $\lambda = m_1(\mu[\lambda])$. The prototypical LE state takes the form 
\beq\label{smooth-LE-typical}\Law(\vN(t))= \bigotimes_{i=1}^N\mu[v(t,i/N)]\eeq for a continuous function $v(t,\cdot):\unit\to\R$, where $\otimes$ denotes taking a product of measures. Thus, the joint distribution $\{v_i(t)\}_{i\in\idxsetx}$ is fully determined by $\mu[v(t,x)]$, since the random variables $v_i(t)$, $i\in\idxsetx$ are independent and approximately distributed according to $\mu[v(t,x)]$ when $N\gg1$, $\epsilon\ll1$. Now, define $\hat f(v):=\mu[v](f)$, the expectation of $f$ under $\mu[v]$. Note that under~\eqref{smooth-LE-typical}, we have \beq\label{proto-Ef}\E[v_i]= v(t,i/N),\quad \E f(v_i) = \hat f(\E v_i).\eeq If there is a pmf $p$ dominating the measure family $\mu[\cdot]$ (see~\cite{kat21} for the details), then $\hat f$ is finite and continuous for any $f\in L^1(p)$. Moreover, by continuity of $v(t,\cdot)$ and $\hat f$, we can take mesoscopic averages of the equality $\E f(v_i)=\hat f(\E v_i)$ to conclude that $\eEf$ is satisfied. Thus, for prototypical LE states, $\hat f$ exists thanks to the fact that the marginals $\Law(v_i)$ belong to a single mean-parameterized measure family.


Of course,~\eqref{smooth-LE-typical} is an idealized situation, and for general interacting particle systems we should not expect $\Law(\vN(t))$ to be an exact prototypical LE state. But the prototypical LE state --- in particular the equalities~\eqref{proto-Ef} --- serve as inspiration for our definition of \emph{smooth} LE states:
\begin{definition}[Smooth LE State~\cite{kat21}]
We say a process $\vN$ has a \emph{\textbf{smooth}} LE state if~\eqref{V} and the following hold for each $t>0$ (dependence on $t$ omitted below):
\begin{align*}
\label{Ep}&\text{There exists a continuous $v:\unit\to\R$ such that}\tag{E$\,'$}\\
&\lim_{N\to\infty}\E[ v_{Nx+k}]= v(t,x),\quad\forall \,x\in\unit, \,k=0,1,2,\dots\text{fixed}\\ 
\label{Efp}&\text{For all ``suitable" }f,\text{ there exists a continuous }\hat f\text{ such that  }\tag{Ef$\,'$}\\ &\E f( v_{Nx+k})\stackrel{N}{\approx}\hat f\l(\E v_{Nx+k}\r), \quad \forall \,x\in\unit,\,k=0,1,2,\dots\text{ fixed}.
\end{align*} 
\end{definition}
The notation $A\stackrel{N}{\approx}B$ means $|A-B|\to0$ as $N\to\infty$. 
\begin{remark} The class of ``suitable" functions $f$ for which $\eEfp$ is satisfied will depend on the LE state. For the prototypical LE state~\eqref{smooth-LE-typical} with dominating pmf $p$, this class consists of the functions $f\in L^1(p)$. 
\end{remark}
By contrast, 
\begin{definition}[Rough LE State~\cite{kat21}]
We say $ \vN$ has a \emph{\textbf{rough}} LE state if $\eV$, $\eE$, and $\eEf$ hold for each $t>0$, while 
 $\eEp$ and $\eEfp$ do \emph{\textbf{not}}. \end{definition} 
To explain the reason for the names ``smooth" and ``rough", we define ``smoothly" and ``roughly" varying as follows.
\begin{definition}[\cite{kat21}]
We say a sequence of vectors $\mbf u_N\in\R^N$, $N=1,2,\dots$ is \emph{\textbf{smoothly}} varying in a neighborhood of $x$ if for any finite $R\geq 0$ we have
$$\lim_{N\to\infty}\max_{|i-Nx|\leq R}|u_{i+1}-u_i| = 0.$$ Otherwise, ${\mbf u}_N $ is \emph{\textbf{roughly}} varying. 
\end{definition}
It is straightforward to see that if $\vN$ has a smooth LE, then the expectations $(\E v_i)_{i=1}^N$ and $(\E f(v_i))_{i=1}^N$ are smoothly varying in the neighborhood of each $x$.

We have already shown in Figures~\ref{fig:E},~\ref{fig:V}, and~\ref{fig:Ef} that $\eV$, $\eE$, and $\eEf$ are satisfied for the Metropolis $\wN$ process. Let us now show that $\eEp$ and $\eEfp$ are not satisfied. It is sufficient to show $(\E w_i)_{i=1}^N$, and $(\E f(w_i))_{i=1}^N$ for some $f$, are roughly varying. Consider Figure~\ref{fig:gazon}, which depicts the observable expectations $\E w_i$, $\E w_i^2$, and $\E J(w_i)$. We see that $\E w_i$ and $\E w_i^2$ are roughly varying, since the rough variation persists as $N$ increases (bottom panel). It also persists as time evolves (top panel). Thus, we have confirmed that $\wN$ has a rough LE state. This is itself an interesting fact; it shows that the rough LE state discovered in~\cite{kat21} is not an isolated phenomenon. 

\begin{figure}
\begin{subfigure}[b]{\textwidth}
\includegraphics[width=\textwidth]{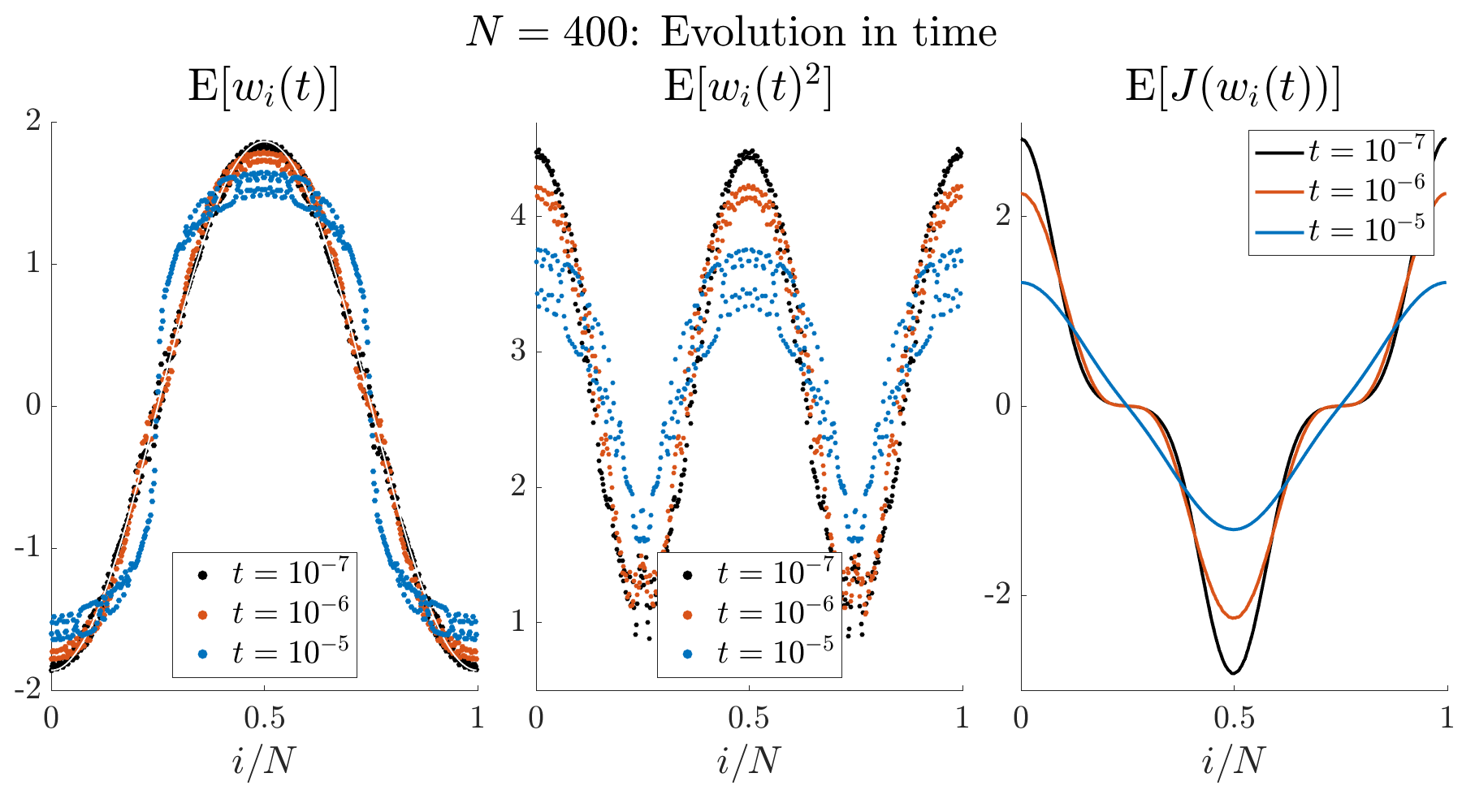}
\vspace{-0.6cm}
\caption{Here we show the expectations of three observables of the $\wN$ process at different points in time, with $N=400$. In the left panel, the thin white line is the initial expected profile, taken to be $\E[w_i(0)]=\mathrm{const.}\times\cos(2\pi i/N)$ exactly. Despite this smooth initial condition, we see that the $\E[w_i(t)]$ profile roughens over time. The $\E[w_i^2]$ profile is also roughly varying, although interestingly enough, $\E[J(w_i)]$ is smoothly varying.}
\label{fig:gazon-time-evoln}
\end{subfigure}
\vspace{-0.9cm}
\begin{subfigure}[b]{\textwidth}
\includegraphics[width=\textwidth]{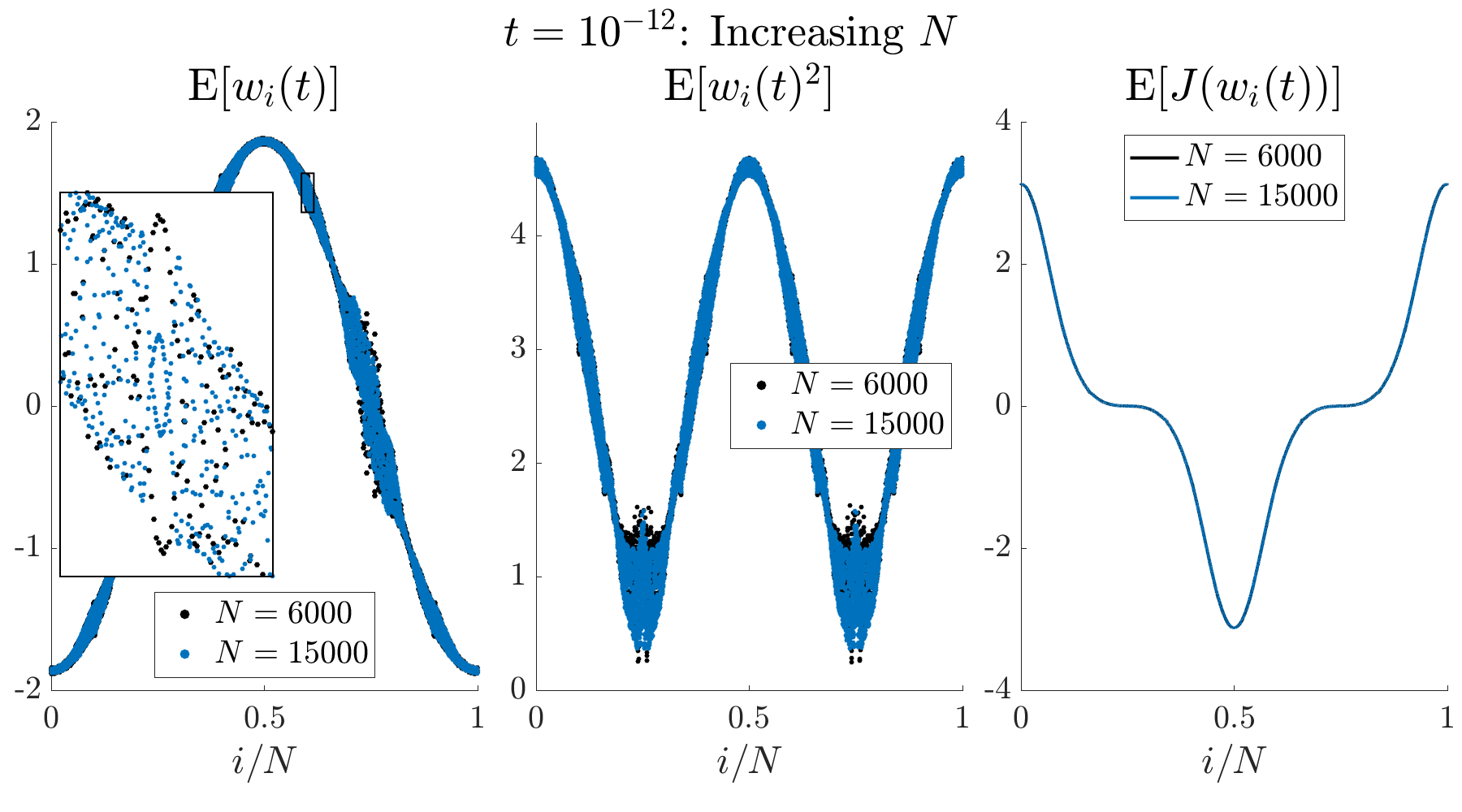}
\caption{The rough variation observed in (a) persists also as $N$ increases, confirming that $\wN$ does \emph{not} have a smooth LE state. }
\label{fig:gazon-N}
\end{subfigure}
\caption{}
\label{fig:gazon}
\vspace{-0.5cm}
\end{figure}

Based on these observables, we see that the qualitative properties of the local equilibrium state of the Metropolis process $\wN$ are very similar to those of the Arrhenius $\wN^\arr$. For the Arrhenius process, the points $(i/N, \E w_i^\arr)$ also form a ``cloud'' with well-defined boundaries, and which narrows near integer values of the range. In addition, despite the fact that both processes have a rough LE, the key functions $f$ (whose corresponding $\hat f$ arises in the PDE) has the property that $(\E f(w_i))_i$ is smoothly varying in both cases. These qualitative similarities between the Arrhenius and Metropolis LE are interesting because as we will show at the end of this section, there is a key difference between the local equilibrium measures of the two processes: for the Arrhenius process, $\Law(\wN^\arr)$ is induced by a local Gibbs distribution on $\zN^\arr$, whereas for the Metropolis process, $\zN$ does \emph{not} follow a local Gibbs distribution. 
\begin{remark}The phenomenon of narrowing near the integers is explained in~\cite{kat21} for the Arrhenius process, but the explanation relies on the local Gibbs assumption, not valid for the Metropolis process.\end{remark}

Let us return to our main goal: computing $\hat f$ for $f=J$. Why might we expect $\hat f$ to exist for a rough LE state like that of the Metropolis $\wN$? There cannot possibly be a mean-parameterized measure family describing each $\Law(w_i)$, because this would imply $\E[J(w_i)]$ can be expressed as a function of $\E[w_i]$. Plotting the former against the latter confirms this is not the case (figure not shown). To answer the question, it is insightful to return to $\wN^\arr$, observed in~\cite{kat21} to have a rough LE state. In that paper, we first confirmed that the distributions $\Law(w_i^\arr)$ are the pmfs induced by $\Law(\zN^\arr)=\rho[\bml]$, the local Gibbs measure defined in~\eqref{gibbs-measure}. We then used this explicit knowledge to show that while $\Law(w_i^\arr)$ is \emph{not} mean-parameterized, we \emph{do} have that
\beq\label{meso-measure-av}\bar \mu_\idxsetx:=\frac{1}{2N\epsilon}\sum_{\argidxsetx i}\Law(w_i^\arr) \approx \mu[\E\ol{\bf w}^\arr_\idxsetx]\eeq for some mean-parameterized family $\mu[\cdot]$. As a result, defining $\hat f(\omega)=\mu[\omega](f)$, we see that $$\E \bar f({\mbf w}^\arr_\idxsetx)=\bar\mu_\idxsetx(f) \approx \mu[\E\ol{\bf w}^\arr_\idxsetx](f)=\hat f(\E\ol{\bf w}^\arr_\idxsetx).$$ The first equality uses linearity of expectation with respect to measures; e.g. if $2N\epsilon=2$ and $\mu_i=\Law(w_i^\arr)$, we are using that $(\int fd\mu_1 + \int fd\mu_2 )/2 = \int fd(\mu_1+\mu_2)/2$. 
\begin{remark}
The Arrhenius LE state shows that the PDE-relevant information contained in the joint law of $\{v_i(t)\}_{i\in\idxsetx}$, is the measure $\bar \mu_\idxsetx$. 
\end{remark}
Due to the qualitative similarity between the LE state of the Arrhenius $\wN^\arr$ and the Metropolis $\wN$, we speculate that the reason for the existence of $\hat J$ is the same for the two LE states: there is some parameterized measure family to which mesoscopic averages of $\Law(w_i)$ all belong. This is supported by Figure~\ref{fig:Ef} confirming $\eEf$ both for $f(w)=J(w)$ and $f(w)=w^2$. To compute $\hat J$ explicitly, however, we would need to know this measure family. But our only guess is the family induced by a local Gibbs product measure on $\Law(\zN)$, and we will now show that this guess is incorrect. 
\subsection{Local Gibbs Approximation: False but Numerically Useful} \label{subsec:gibbs}
Recall from~\eqref{gibbs-measure} the form of the local Gibbs product measure $\rhogibbsNd$. By completing the square in the exponent, we can also write 
$$\rhogibbsNd = \otimes_i\rho[\lambda_i],\quad\text{where }\quad\rho[\lambda](n) = \frac{e^{-K(n-\lambda)^2}}{\calZ(\lambda)},\quad n\in\Z.$$ Here, $\calZ(\lambda)=\sum_{m=-\infty}^\infty \e(-K(m-\lambda)^2).$ To show that $\Law(\zN(t))$ is not a local Gibbs distribution for any $\bml$, consider the following specially chosen observables: $$f_i^\pm(\mbf z) = \e\l(\pm 2K(z_{i-1}-2z_i + z_{i+1})\r).$$ We will compare the expectation of the $f_i^\pm$ under the local Gibbs measure and under the true measure. 
Now, we showed in Section 6.1 of~\cite{kat21} that
\beq\label{exp-expect-gibbs}\rho[\lambda](e^{cKz}) = \e(c^2K/4 + cK\lambda)\frac{\calZ(\lambda+c/2)}{\calZ(\lambda)}.\eeq Using this formula, the fact that $\calZ$ is a period 1 function, and the independence of the $z_i$ under the product measure $\rhogibbsNd$, we compute
\beqs\label{f-expect-gibbs}
\rho[\bml]&(f_i^\pm) = \e\l(6K\pm 2K(\lambda_{i+1}-2\lambda_i + \lambda_{i-1})\r),\\
&\implies\rho[\bml]\l( f_i^+\r)\times\rho[\bml]\l(f_i^-\r) \equiv e^{12K}\eeqs for all $i$, and regardless of $\bml$. 
Thus, we can confirm that $\Law(\zN(t))$ is not a local Gibbs measure by showing that under the true measure,
\beq\label{log-prod}\log\big(\E\l[ f_i^+(\zN(t))\r]\times\E \l[f_i^-(\zN(t))\r]\big)\neq 12K.\eeq This is shown in Figure~\ref{fig:no-gibbs}, with $K=1$. Interestingly, the constant $12K$ seems to be a tight lower bound for the lefthand side of~\eqref{log-prod}. \\
\begin{wrapfigure}{r}{0.46\textwidth}
\vspace{-0.4cm}
\centering
\includegraphics[width=0.46\textwidth]{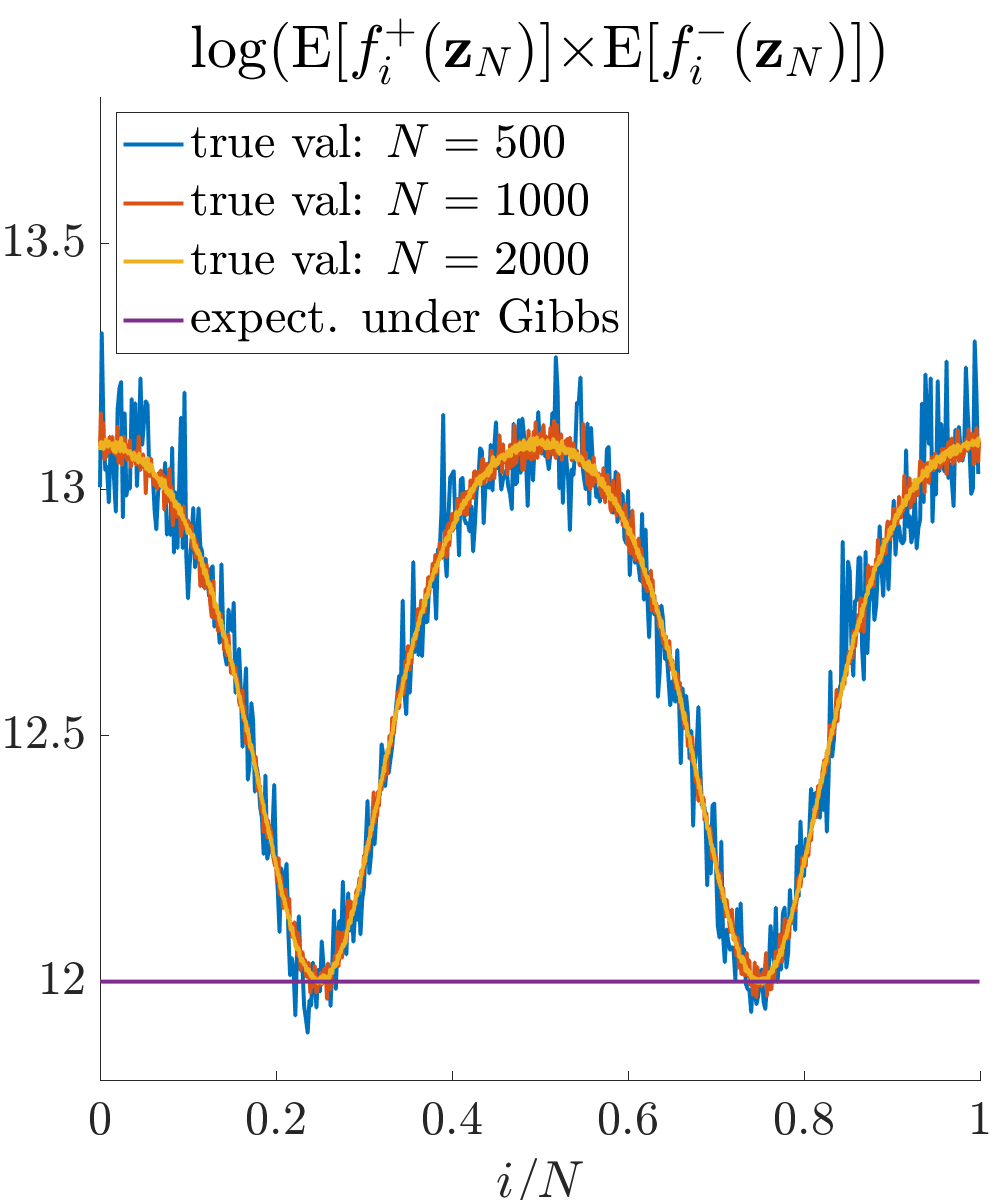}
\caption{This figure confirms~\eqref{log-prod}, which implies that for the Metropolis process, $\Law(\zN(t))$ is not given by a local Gibbs measure. Here, $K=1$, and we see that $12K=12$ is a lower bound for the quantity on the lefthand side of~\eqref{log-prod}.}
\label{fig:no-gibbs}
\end{wrapfigure}
\noindent\textbf{Useful Numerical Estimate.} As we will explain in Section~\ref{subsec:hatJ}, the estimate $\hat J_\gibbs$ of $\hat J$ obtained by assuming $\Law(\zN)=\rho[\bml]$ is a useful baseline estimate. To obtain $\hat J_\gibbs$ we must be able to write $\rho[\bml](\fxnbar J w)$ as a function of $\E\varbar w$. To do so, we first note that $\lambda_i$ must be given by $\lambda_i=\lambda(\E z_i)$, where the function $\lambda$ is the inverse of $\lambda\mapsto m_1(\rho[\lambda])$. Now, $\rho[\bml](\fxnbar J w)$ is a function of $\lambda_i$, $i\in\idxsetx$ and therefore it is some further function of the $\E[z_i]$. 

For general $K$, it is unclear whether this further function depends on the $\E[z_i]$ only through $\E\varbar w$, i.e. whether a function sending $\E\varbar w$ to $\rho[\bml](\fxnbar J w)$ exists at all. This is because the function $\lambda$ depends nonlinearly on $\E[z_i]$.  However, when $K$ is small, simplifying approximations make this possible, and one obtains
\beq\label{Jgibbs}\hatJgibbs(\omega) = 2e^{-3K/2}\sinh(K\omega).\eeq See~\cite{gao2020} for the computation of this function. We do not bother obtaining a more exact estimate of $\hat J_\gibbs$ for larger $K$ since the local Gibbs distribution is incorrect. We only need a baseline estimate which will help us compute the true $\hat J$, and it will turn out that the estimate~\eqref{Jgibbs} suits our needs, even for larger $K$. 

We note that the discrepancy observed in~\cite{gao2020} between $\hN$, $N\to\infty$ and the solution to the PDE $h_t=-\partial_x\hatJgibbs(h_{xxx})$, is \emph{not} caused by small $K$ approximations. The simulations in that paper take $K=0.25$, for which $\hatJgibbs$ is a very accurate estimate of the local Gibbs expectation. Rather, the discrepancy is caused by the fact that the local Gibbs distribution is not the correct LE state.


\section{Numerical Implementation}\label{sec:num}
We begin in Section~\ref{subsec:setup} by describing how we simulate the Metropolis dynamics and compute expectations of observables. In Section~\ref{subsec:hatJ}, we explain in more detail how we compute the function $\hatJ$. Finally, in Section~\ref{subsec:PDE-confirm}, we confirm that we computed the function $\hatJ$ correctly: we show that the microscopic processes $\wN$ and $\hN$ converge to the solutions of the PDEs~\eqref{w-PDE} and~\eqref{h-PDE}, respectively, and \emph{not} to the corresponding PDEs with $\hatJgibbs$. 

\subsection{Set Up}\label{subsec:setup}
Since the microscopic dynamics is a Markov jump process, the path $\{\hN(t)\}_{t\geq0}$ is a step function, with $\hN(t) = \mbf h_{k}$ when $t\in [t_k, t_{k+1})$. Therefore, simulating the process in a time interval $[0,T]$ amounts to drawing the pairs $(\mbf h_k, t_k)$, $t_k\leq T$, according to the law of the process $\hN$. We do so using the Kinetic Monte Carlo algorithm (KMC)~\cite{KMC}, presented in Algorithm~\ref{alg:KMC} below. Note that the algorithm uses rates $r^{i,j}$ and rescales time, which is equivalent to using rates $N^4r^{i,j}$. 
\begin{algorithm2e}
\caption{KMC algorithm to simulate Markov jump processes}\label{alg:KMC}
\KwInput{$N$, initial value $\bh\in \Z^N$, macroscopic end time $t$, rates $r^{i,j}$.}
\KwOutput{Jump times $t_1, t_2, t_3,\dots$ and states $\bh_1,\bh_2,\bh_3,\dots$}
$s \gets 0$, $k \gets 0$, $\bh_k\gets \bh$\;
\While{$s \leq N^4t$}{
  $R\gets \sum_{|i-j|=1}r^{i,j}(\bh_k)$\;
  Draw $T\sim\mathrm{Exp}(R)$\;
  Draw $\bh_{k+1}$ from the pdf $p(\bh^{i,j}) = r^{i,j}(\bh_k)/R$\;
  $s \gets s + T$, $t_k\gets N^{-4}s$, $k \gets k+1$
}
\end{algorithm2e}
Given a smooth macroscopic initial condition $h_0(x)$, we initialize KMC with a microscopic height profile $\hN(0)$ drawn from
\beq\label{init:num}\hN(0) \sim \l(\l\lfloor N^3h_0\l(\frac iN\r)\r\rfloor  + \xi_i\r)_{i=1}^{N},\quad \xi_i\sim \text{Bernoulli}(p_i),\eeq where the $\xi_i$ are independent and $p_i$ is the fractional part of $N^3h_0(i/N)$. Thus $\E[h_i]=N^3h_0(i/N)$ exactly. This fact and the independence of the $\xi_i$ ensures that~\eqref{init}, $\beta=0$, is satisfied for $\wN$ and that $N^{-1}\sum_{i=1}^Nh_i(0)$ converges to a constant. 

Most of the quantities we need to compute are observables $\E[f(\wN(t))]$ of the third order FD process $\wN(t)$. We estimate these by drawing $M$ independent initial conditions $\hN^{(k)}(0)$, $k=1,\dots, M$ from the distribution~\eqref{init:num}, evolving them forward using KMC, and then estimating
\beq\label{sample-av}
\E[f(\wN(t))]\approx \frac 1M\sum_{k=1}^Mf(\mbf w^{(k)}_{N}(t)),
\eeq
where $\wN^{(k)}$ is the third order FD of $\hN^{(k)}$. Some observables $f$ we need to compute (such as the current observable $f(w) = \sinh(Kw)$) have extremely high variance, and the number of samples $M$ needed to sufficiently reduce the sample variance of~\eqref{sample-av} is intractable. For such observables, we can reduce the variance further by integrating over a small time window:
\beq\label{sample-time-av}
\E[f(\wN(t))]\approx \frac{1}{\Delta}\int_{I_{t,\Delta}}\E[f(\wN(s)]ds \approx \frac 1M\sum_{k=1}^M\frac{1}{\Delta}\int_{I_{t,\Delta}}f(\mbf w^{(k)}_{N}(s))ds,\quad \Delta\ll1.\eeq Here, $I_{t,\Delta}$ is some time interval of length $\Delta$ containing $t$. 
We can compute the time integrals in~\eqref{sample-time-av} exactly since the paths of the process are step functions. See Appendix~\ref{app:num} for justification of the time average approximation to the expectations $\E[f(\wN(t))]$. 
\begin{figure}
\centering
\includegraphics[width=\textwidth]{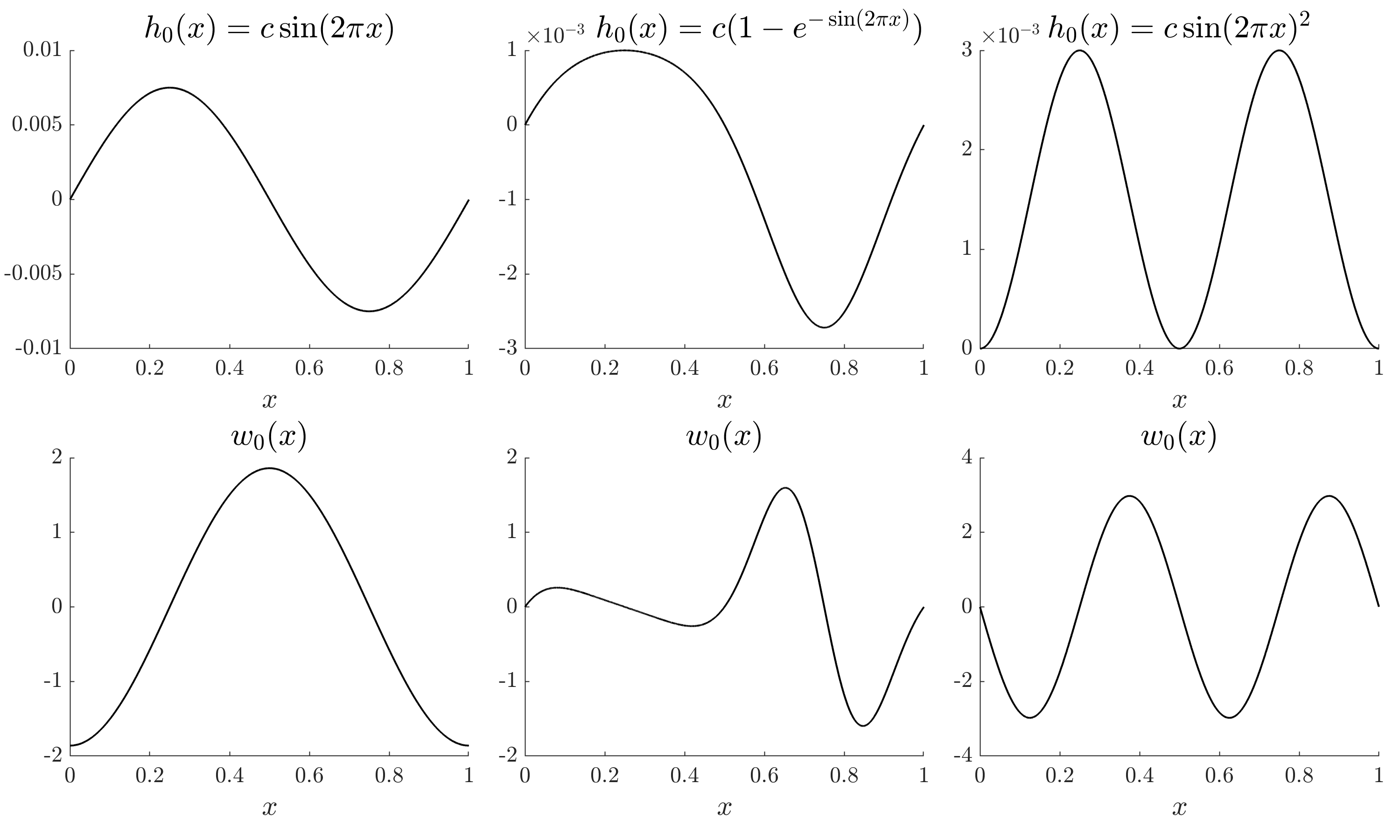}
\caption{Initial Conditions. We take $c=0.0075$ in the first (leftmost) column, $c=0.001$ in the second column, and $c=0.003$ in the third column. The initial condition $h_0(x)=c\sin^2(2\pi x)$ in the third column is reserved for computing $\hat J$.}
\label{fig:init}
\end{figure}

Figure~\ref{fig:init} depicts the initial conditions (ICs) we used in our simulations. The rightmost IC is reserved for computing the function $\hatJ$ in the PDE. We use the other two ICs to confirm the resulting $\hatJ$ gives the correct PDE. 

\subsection{Computation of Current}\label{subsec:hatJ}
In this section, we will describe our procedure to compute $\hatJ$, for which we will use statistics collected from a process with initial condition given by the third column in Figure~\ref{fig:init}, and $K=2$. 

As we mentioned in Section~\ref{subsec:theory}, our general strategy for computing $\hatJ$ is to interpolate the points $(\E\varbar w(t),\, \E \bar J({\bf w}_\idxsetx(t)))$ computed at multiple $(t,x)$. However, we will need to refine this strategy slightly to incorporate new information and to make numerical estimation more convenient. The new information we have is that, as seen in Figure~\ref{fig:gazon}, the expectations $\E J(w_i)$ vary smoothly with $i$. Therefore we will not average the expected current over space. The next modification is to replace expectations at a single time $t$ with expectations integrated over $s\in I_{t,\Delta}$. As explained above, sample average estimates of these time-integrated quantities have much lower variance. In sum, we replace $(\E\varbar w(t),\, \E \bar J({\bf w}_\idxsetx(t)))$ with $(\omega_{\epsilon,\Delta}(t,x), \J_\Delta(t,x))$, where 
 \beqs\label{bar-omega-J}
\omega_{\epsilon,\Delta}(t,x) := \frac1\Delta\int_{I_{t,\Delta}}\E\varbar w(s)ds,
\eeqs
and \beqs\J_{\Delta}(t,x) := \frac1\Delta\int_{I_{t,\Delta}}\E J(w_{\lfloor Nx\rfloor}(s))ds.\eeqs In what follows, we will write $\omega_{\epsilon,\Delta}(t,x)$ and $\J_\Delta(t,x)$ to denote our \emph{sample average estimates} of these quantities, computed as in~\eqref{sample-time-av}. We also sometimes abbreviate notation by writing $\omega_{\epsilon,\Delta}=\omega_{\epsilon,\Delta}(t,x)$ and $\J_{\Delta}=\J_{\Delta}(t,x)$ to denote these quantities at a generic point $(t,x)$. 
 
Next, we need to address an important issue with the strategy of interpolating $(\omega_{\epsilon,\Delta}, \J_\Delta)$ to compute the function $\hatJ$: namely, we need the $x$-coordinates $\omega_{\epsilon,\Delta}$ to span all of $\R$ in order to accurately estimate the value of $\hatJ(\omega)$ for $\omega$ ranging over all of $\R$. This is where the ``baseline estimate" $\hatJgibbs$ come in handy. Conveniently, $\hatJ(\omega)/\hatJgibbs(\omega)$ \emph{rapidly levels out to a constant asymptote as} $|\omega|\to\infty$! Therefore, it will be much simpler to estimate
\beq\label{sigmaK}\sigma(\omega):=\hatJ(\omega)/\hatJgibbs(\omega),\eeq and then obtain $\hatJ$ as $\hatJ=\sigma\hatJgibbs$. We can estimate $\sigma(\omega)$ by \emph{interpolating} the points
 \beq\label{epsDelt-points}\bigg\{\bigg(\omega_{\epsilon,\Delta}(t,x),\;\frac{\J_\Delta(t,x)}{\hatJgibbs(\omega_{\epsilon,\Delta}(t,x))}\bigg)\;\bigg\vert\; x=\frac1N,\frac2N,\dots, 1\bigg\}\eeq inside a bounded domain, and \emph{extrapolating} to a constant outside of it. 
 
 This strategy raises a new issue, which is that both $\J_{\Delta}$ and $\hatJgibbs(\omega_{\epsilon,\Delta})$ approach zero when $\omega_{\epsilon,\Delta}\to0$. We will address this issue shortly. First let us choose suitable parameters $t=t_N$, $\epsilon=\epsilon_N$, $\Delta=\Delta_N$ for each of $N=1000,2000,4000$. We will then plot three curves of points $\{(\omega_{\epsilon,\Delta},\;\J_\Delta/\hatJgibbs(\omega_{\epsilon,\Delta}))\}$ corresponding to these three values of $N$, to ensure that the curves have converged. 

\textbf{Step 1: Choose $t$.} We observe that $\sup_i|\E w_i(t)|$ decreases with time, so for the purpose of generating points $\omega_{\epsilon,\Delta}(t,x)$ which span a large interval, it is better to take $t$ small. On the other hand, $t$ must be sufficiently large that the process has had time to \emph{locally equilibrate}. As $N\to\infty$, the ``burn-in" time $t_N$ until local equilibration --- i.e. until we can expect the crucial condition $\eEf$ of a rough LE state to be satisfied --- should go to zero. In other words, equilibration occurs instantaneously when $N\to\infty$. However, for finite $N$ we must be careful to wait sufficiently long so as to avoid collecting statistics $(\omega_{\epsilon,\Delta}(t,x), \J_{\Delta}(t,x))$ from a pre-LE distribution. In order to determine whether a given time $t$ is past burn-in for a fixed $N$, we do the following heuristic test: first, we check that the points $(\omega_{\epsilon,\Delta}(t,x), \J_{\Delta}(t,x))$, $x=i/N$, all lie on a single curve, i.e. they pass the straight line test. Second, we check that for $t'>t$, the corresponding points lie on the same curve as at time $t$.

\textbf{Step 2: Choose $\epsilon,\Delta$}. We choose appropriate values $\epsilon=\epsilon_N$ and $\Delta=\Delta_N$ as follows: we plot the points~\eqref{epsDelt-points} (with $t=t_N$) for a range of $\epsilon$ and $\Delta$, and look for $\epsilon_N$, $\Delta_N$ which lead to curves which are neither too biased compared to the curve corresponding to the smallest $\epsilon,\Delta$, nor too noisy. Figure~\ref{fig:sig-eps-Delt} depicts these curves for fixed $N=4000$ and several values of $\epsilon,\Delta$. We omit points in the set~\eqref{epsDelt-points} for which $|\omega_{\epsilon,\Delta}|<\delta_0$ for a $\delta_0\ll1$. We see in the figure that the effect of varying $\Delta$ is much less significant than the effect of varying $\epsilon$. For $N=4000$, we take $\epsilon=0.0015$ and $\Delta=4\times 10^{-10}$. 
\begin{figure}
\centering
\includegraphics[width=0.8\textwidth]{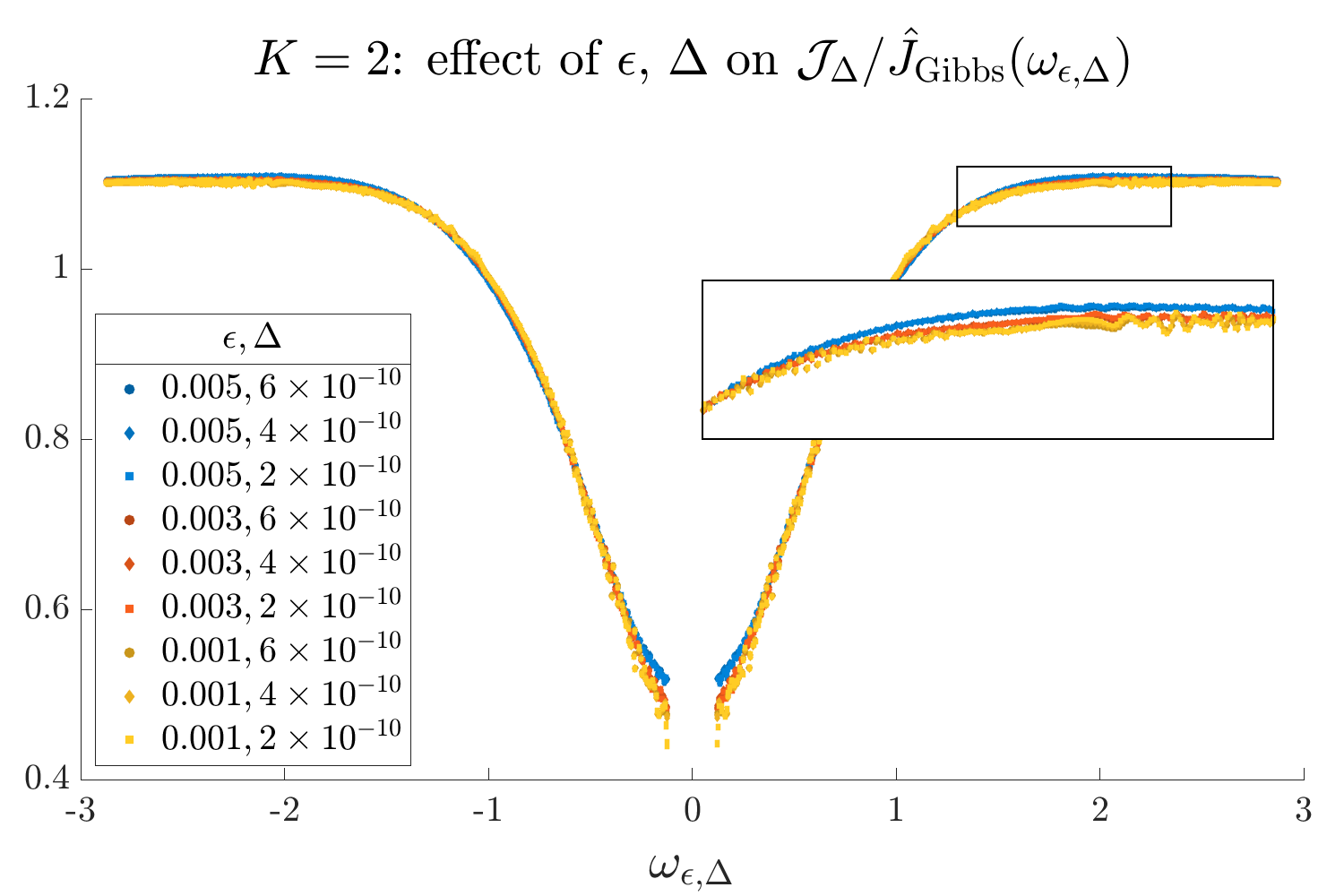}
\caption{Effect of varying $\epsilon$ and $\Delta$ on the curve of points $(\omega_\epsDel, \J_\Delta/\hatJgibbs(\omega_\epsDel))$ at fixed $N=4000$, and $K=2$. We see that varying $\epsilon$ has a more significant effect than does varying $\Delta$. Based on this plot, we choose $\epsilon_N=0.0015$ and $\Delta_N=4\times10^{-10}$.}
\label{fig:sig-eps-Delt}
\vspace{-0.5cm}
\end{figure}
Using this procedure for $N=1000,2000$, we take $\epsilon_N=0.003,0.002$, respectively, and $\Delta_N=4\times10^{-10}$ for both. 

\textbf{Step 3: ``Fill in" the curve near zero.} Having identified $\epsilon_N$ and $\Delta_N$, we next ``fill in'' the curve of points~\eqref{epsDelt-points} in the neighborhood $|\omega|<\delta_0$. We do so using the numerical observation that $\sigma$ has a local (and global) minimum at zero. This implies that for small values of $\omega$ we should have $\sigma(\omega)\approx a+ b\omega^2$ for some values $a,b$. We find optimal $a=a_N$, $b=b_N$ for each $N$ by solving
\beq\label{ab}
(a_N, b_N) = \arg\min_{(a,b)}\sum_{|\omega_{\epsilon,\Delta}|<\delta_1}\bigg(\J_{\Delta} - (a + b\,\omega_{\epsilon,\Delta}^2)\hatJgibbs(\omega_{\epsilon,\Delta})\bigg)^2,
\eeq where $\epsilon=\epsilon_N$, $\Delta=\Delta_N$, and the sum is over all points in the set~\eqref{epsDelt-points} such that $|\omega_{\epsilon,\Delta}|<\delta_1$. Here, $\delta_1$ is small but greater than $\delta_0$. We take $\delta_1>\delta_0$ in order to obtain a smoother transition between the quadratic approximation near the origin and the remaining curve.  

Next, we visually confirm that the filled in curves converge as $N$ increases. This is shown in Figure~\ref{fig:converge}. 
\begin{figure}
\centering
\includegraphics[width=0.8\textwidth]{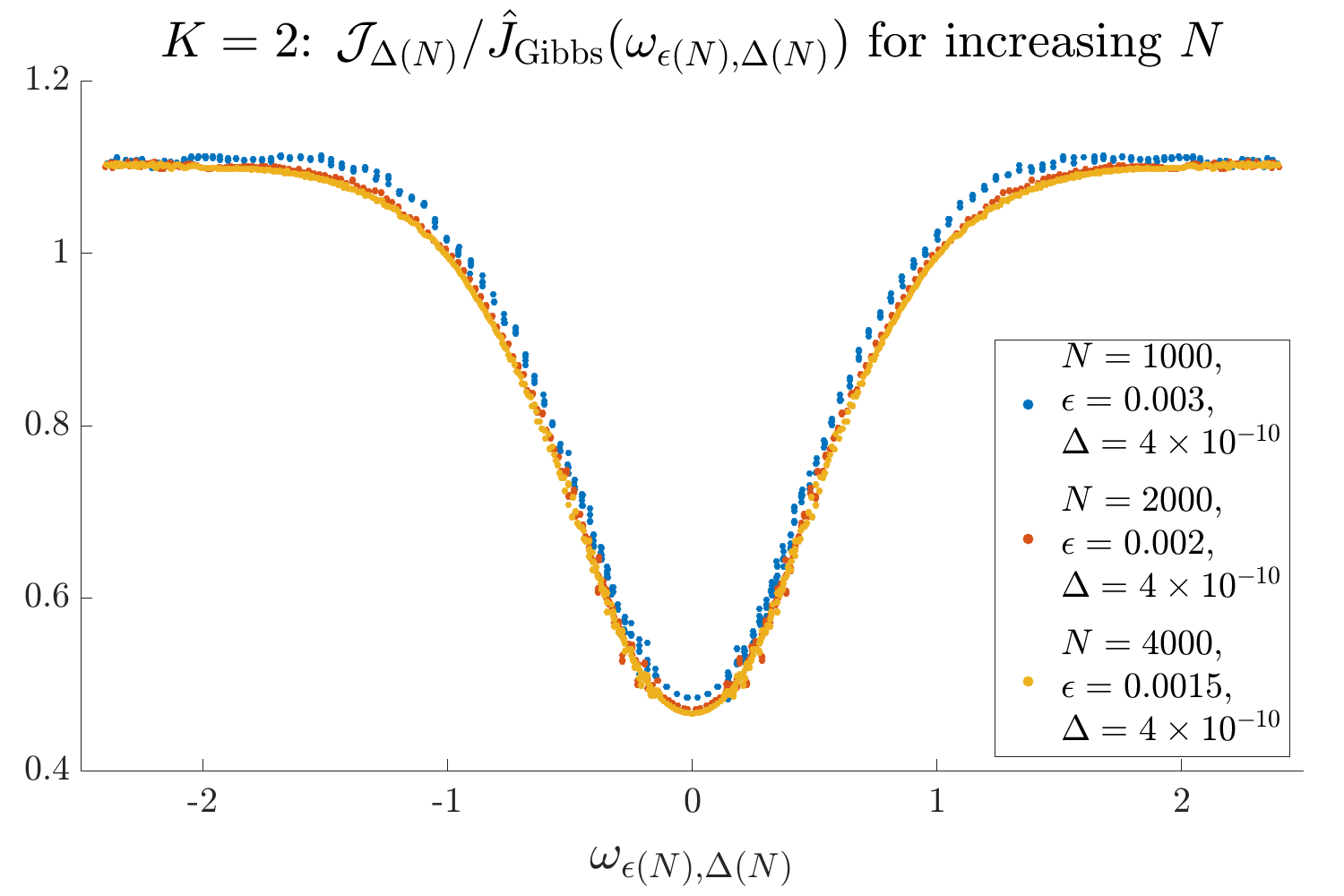}
\caption{Convergence with $N$ of the curve of points $(\omega_\epsDel, \J_\Delta/\hatJgibbs(\omega_\epsDel))$, where $\epsilon=\epsilon_N$, $\Delta=\Delta_N$ are chosen using the procedure described in the text. Near the origin, we substitute $ \J_\Delta/\hatJgibbs(\omega_\epsDel)$ by $a_N+b_N\omega_\epsDel^2$, where $a_N,b_N$ minimize~\eqref{ab}.}
\label{fig:converge}
\end{figure}
Finally, we use the $N=4000$ curve to compute the function $\sigma$, by fitting a smoothing spline to it. We fit the spline inside a bounded interval (e.g. $[-2.5,2.5]$ for $K=2$) by calling MATLAB's \texttt{fit} routine with the option ``smoothingspline". This routine implements the following minimization:
\beqs\label{spline}
\sigma= \arg\min_{\text{splines } s}&\;\;\lambda\!\!\sum_{|\omega_\epsDel|>\delta_0}\l(\J_\Delta/\hatJgibbs(\omega_\epsDel)- s(\omega_\epsDel)\r)^2 \\&+ \lambda\!\!\sum_{|\omega_\epsDel|<\delta_0}\l(a+b\omega_\epsDel^2 - s(\omega_\epsDel)\r)^2 + (1-\lambda)\int s''(x)^2dx.\eeqs The coefficient $\lambda\in (0,1)$ is a smoothing parameter. We then extrapolate the spline to be constant outside the bounded interval using MATLAB's \texttt{fnxtr}. 

Figure~\ref{fig:allUFOs} depicts the result of implementing this procedure for a range of $K$ values. For each $K$, we plot the points~\eqref{epsDelt-points} generated from the $N=4000$ process and using the chosen $\epsilon=\epsilon_N$, $\Delta=\Delta_N$ as described above. These curves are shown in color. They appear smoother near the origin because for $|\omega_\epsDel|\ll1$, we replace $(\omega_\epsDel, \J_\Delta/\hatJgibbs(\omega_\epsDel))$ with $(\omega_\epsDel, a+b\omega_\epsDel^2)$. The curves are overlayed with their spline approximations in black. 
\begin{figure}
\centering
\includegraphics[width=0.6\textwidth]{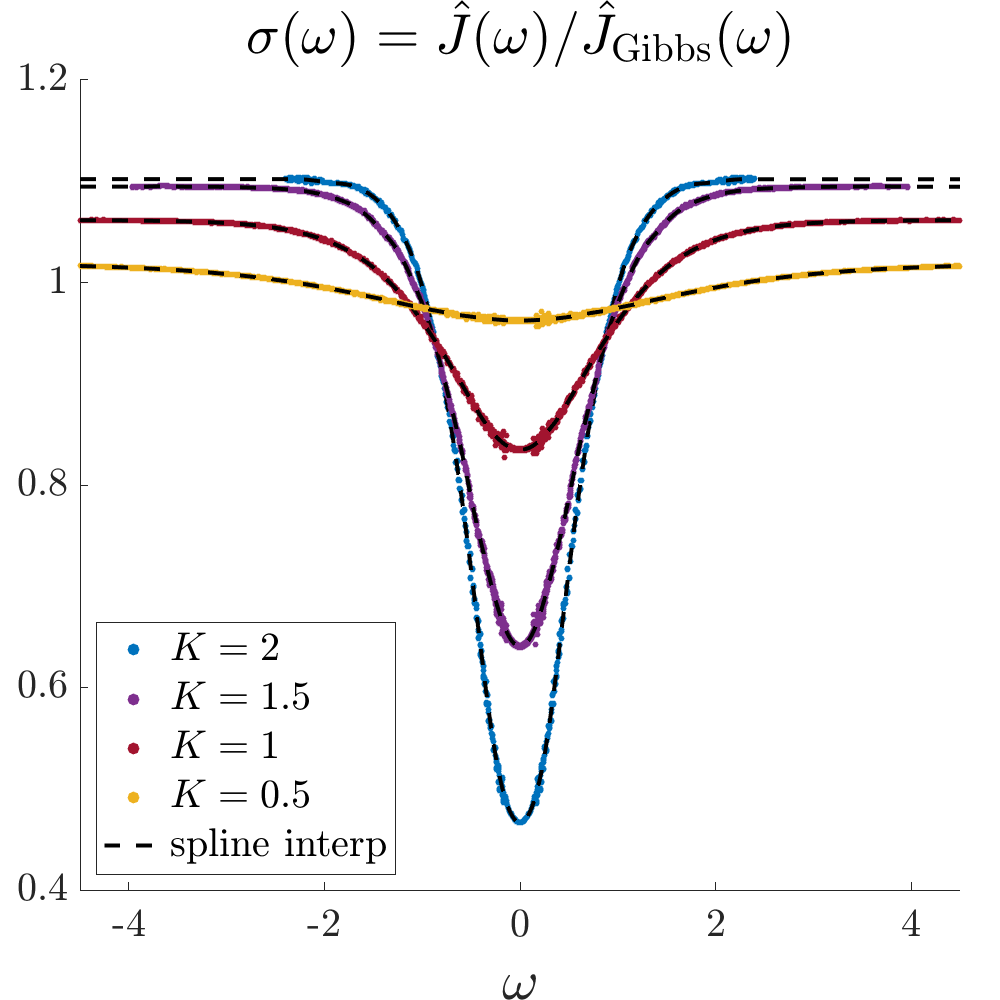}
\caption{Function $\sigma=\sigma_K$ for different values of $K$. The black dotted lines are the smoothing spline interpolations computed as in~\eqref{ab},~\eqref{spline}. Near $\omega=0$, we plot $(\omega, a+b\omega^2)$ rather than $(\omega_{\epsilon,\Delta}, \J_\Delta/\hatJgibbs(\omega_{\epsilon,\Delta}))$. Note that the functions $\sigma$ are all even, nondecreasing on $\R^{+}$, and bounded above and below by positive constants. Note also that $\sigma$ approaches the constant 1 as $K$ decreases.}
\label{fig:allUFOs}
\vspace{-0.5cm}
\end{figure}

There are two notable features of this family of corrections $\sigma=\sigma_K$. First, the corrections approach the constant 1 as $K$ decreases, in line with the observation in~\cite{gao2020} that the PDE derived under the local Gibbs assumption is very nearly accurate at low $K$. Second, we note that $\sigma$ is even, nondecreasing on $\R^{+}$, and bounded above and below by positive constants for all $K$. These qualitative observations will enable us to extend the analysis of the PDE~\eqref{introPDE-gibbs} done in~\cite{gao2020}, to the analysis of the corrected PDE~\eqref{introPDE}.

\subsection{Convergence to PDE Solution}\label{subsec:PDE-confirm}
We will take $K=2$ in our verification of the PDE, and the initial height profiles $h_0$ depicted in the first two columns of Figure~\ref{fig:init}. We numerically solve the PDE 
\beq\label{J-PDE}\begin{cases}
h_t = -\partial_x\l(\sigma(h_{xxx})2{e}^{\frac{-3K}{2}}\sinh(Kh_{xxx})\r),&\quad t>0,x\in (0,1)\\ h(0,x)=h_0(x),&\quad x\in (0,1)\end{cases},\eeq with $\sigma$ computed as described in Section~\ref{subsec:hatJ}. For comparison, we also numerically solve the PDE without the correction, letting $\tilde h$ denote the solution to
\beq\label{tilde-J-PDE}\begin{cases}
\tilde h_t = -\partial_x\l(2{e}^{\frac{-3K}{2}}\sinh(K\tilde h_{xxx})\r),&\quad t>0,x\in (0,1)\\ \tilde h(0,x)=h_0(x),&\quad x\in (0,1).\end{cases}\eeq  
We solved the PDEs by discretizing the spatial differential operators, and evolving the resulting ODE forward using MATLAB's \texttt{ode15s}, which is designed for stiff differential equations. Our primary interest is to confirm that the PDE~\eqref{J-PDE} is the correct limit of the microscopic dynamics. We will therefore study pointwise convergence of $\hN$ rather than the hydrodynamic convergence of Definition~\ref{def:hydro}. We will see that $$\E[h_N(t,x)]:=\E h_{\lfloor Nx\rfloor}(t)\to h(t,x),\quad N\to\infty,\;\forall \; (x,t),$$ \emph{with no spatial averaging required}. This is in stark contrast to the $\wN$ process, for which $\E w_{\lfloor Nx\rfloor}(t)$ does not converge at all. We will also confirm that $$\E\varbar w(t)\to h_{xxx}(t,x), \quad N\to\infty,\;\forall \; (x,t),$$where $\epsilon=\epsilon(N)$ is chosen using the procedure described in the verification of $\eE$ in Section~\ref{subsec:theory-numeric}. Note that $\eE$ only verified $\E\varbar w(t)$ has \emph{some} limit, whereas now we verify this limit is the third derivative of the solution to~\eqref{J-PDE}.

We start with the initial condition $h_0(x)=c(1-e^{-\sin(2\pi x)})$. The left panel of Figure~\ref{fig:exp-evoln} depicts the evolution of $\E [h_N(t,\cdot)]$ in time for $N=500$, as well as the evolution of $h(t,\cdot)$ and $\tilde h(t,\cdot)$, where $h$, $\tilde h$ are solutions to~\eqref{J-PDE} and~\eqref{tilde-J-PDE}, respectively. The right panel shows the evolution of $\E\varbar w(t)$ in comparison to $h_{xxx}$ and $\tilde h_{xxx}$. We see that there is a nontrivial qualitative difference between the two macroscopic evolutions, and that the evolution of the microscopic process clearly follows the PDE~\eqref{J-PDE} rather than the PDE~\eqref{tilde-J-PDE}. This shows that the correction $\sigma$ is necessary to capture the correct dynamics. 

The left panel of Figure~\ref{fig:exp-converge} depicts $\E[h_N(t,\cdot)-h_N(0,\cdot)]$ for $t=2\times 10^{-8}$ and $N=250,500,1000$, where the $N=\infty$ curve is $h(t,\cdot)-h_0$. We plot the time increment of $\E h_N$ rather than $\E h_N$ itself, in order to better see the convergence (at this $t$, the process $h_N(t)$ is still very close to $h_N(0)$). The right panel of the figure depicts $\E\varbar w(t)$ for increasing $N$, with $N=\infty$ representing $h_{xxx}(t,\cdot)$. The panels confirm that $\E[h_N(t,\cdot)]$ is converging to $h(t,\cdot)$, since $\E[h_N(0)]$ converges to $h_0$ by design, and that $\E\varbar w(t)$ is converging to $h_{xxx}(t,\cdot)$.
\begin{figure}
\begin{subfigure}[b]{\textwidth}
\includegraphics[width=\textwidth]{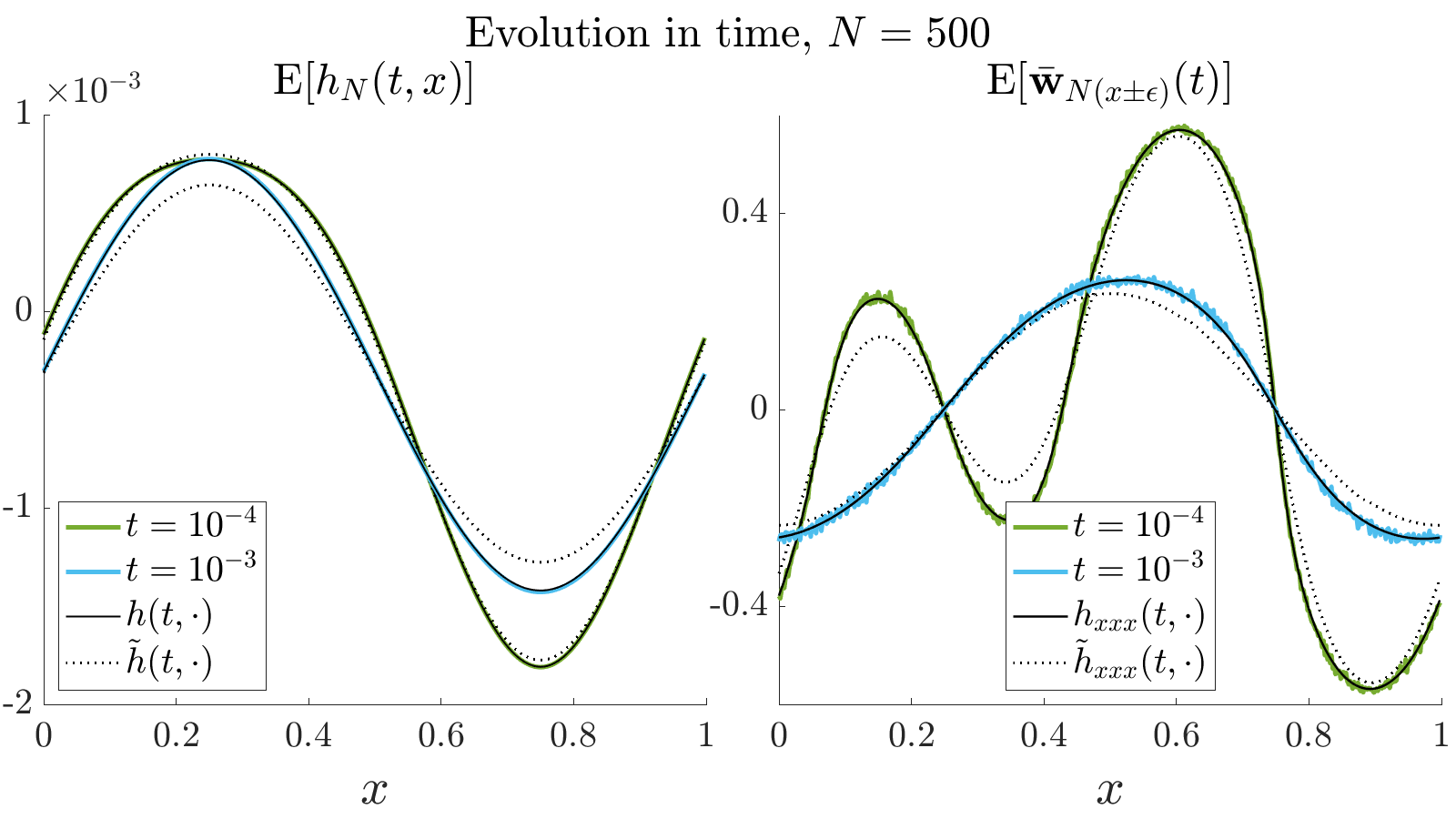}
\caption{Exponential initial condition.  Left: evolution of $\E h_N$ in time, for $N=500$. The evolution is compared to that of $ h(t,\cdot)$ and $\tilde h(t,\cdot)$. It is clear that the evolution of $\E h_N$ follows that of $h$ rather than $\tilde h$. Right: evolution of $\E\varbar w(t)$ in time for $N=500$, compared to the evolution of $h_{xxx}$ and $\tilde h_{xxx}$. We see that $\E\varbar w(t)$ follows $h_{xxx}$ rather than $\tilde h_{xxx}$. Note the qualitative difference between the two macroscopic profiles.}
\label{fig:exp-evoln}
\end{subfigure}
\begin{subfigure}[b]{\textwidth}
\includegraphics[width=\textwidth]{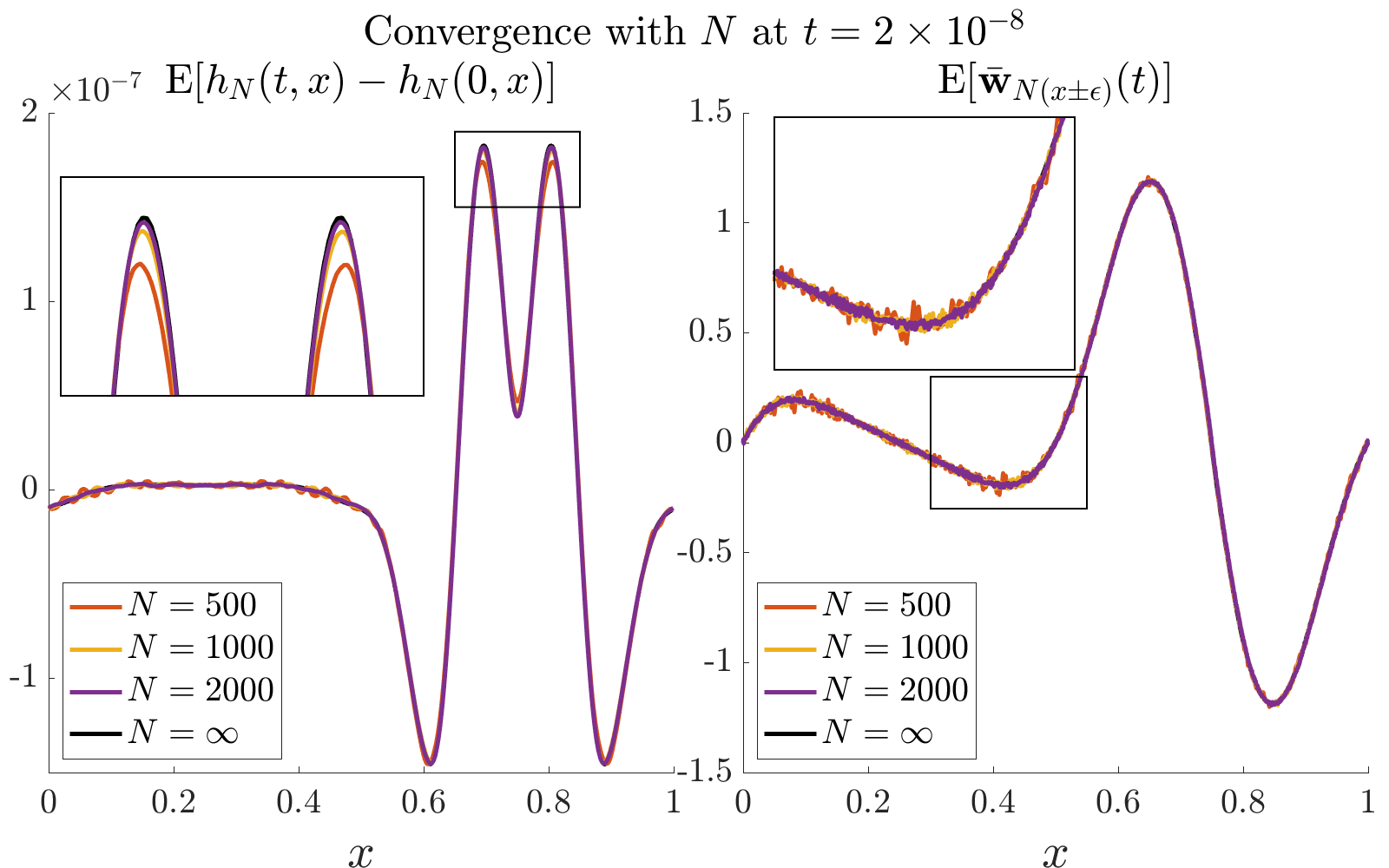}
\caption{Exponential initial condition. Convergence of $\E[h_N(t,\cdot)-h_N(0,\cdot)$ to $h(t,\cdot)-h_0$, and $\E\varbar w$ to $h_{xxx}$, as $N\to\infty$.}
\label{fig:exp-converge}
\end{subfigure}
\caption{}
\end{figure}

Figures~\ref{fig:sin-evoln} and~\ref{fig:sin-converge} are analogous, but for the sinusoidal initial condition. For this IC, the qualitative differences between the two macroscopic evolutions~\eqref{J-PDE} and~\eqref{tilde-J-PDE} are not as significant but again, it is clear that the microscopic process converges to the solution of~\eqref{J-PDE}.
\begin{figure}
\begin{subfigure}[b]{\textwidth}
\includegraphics[width=\textwidth]{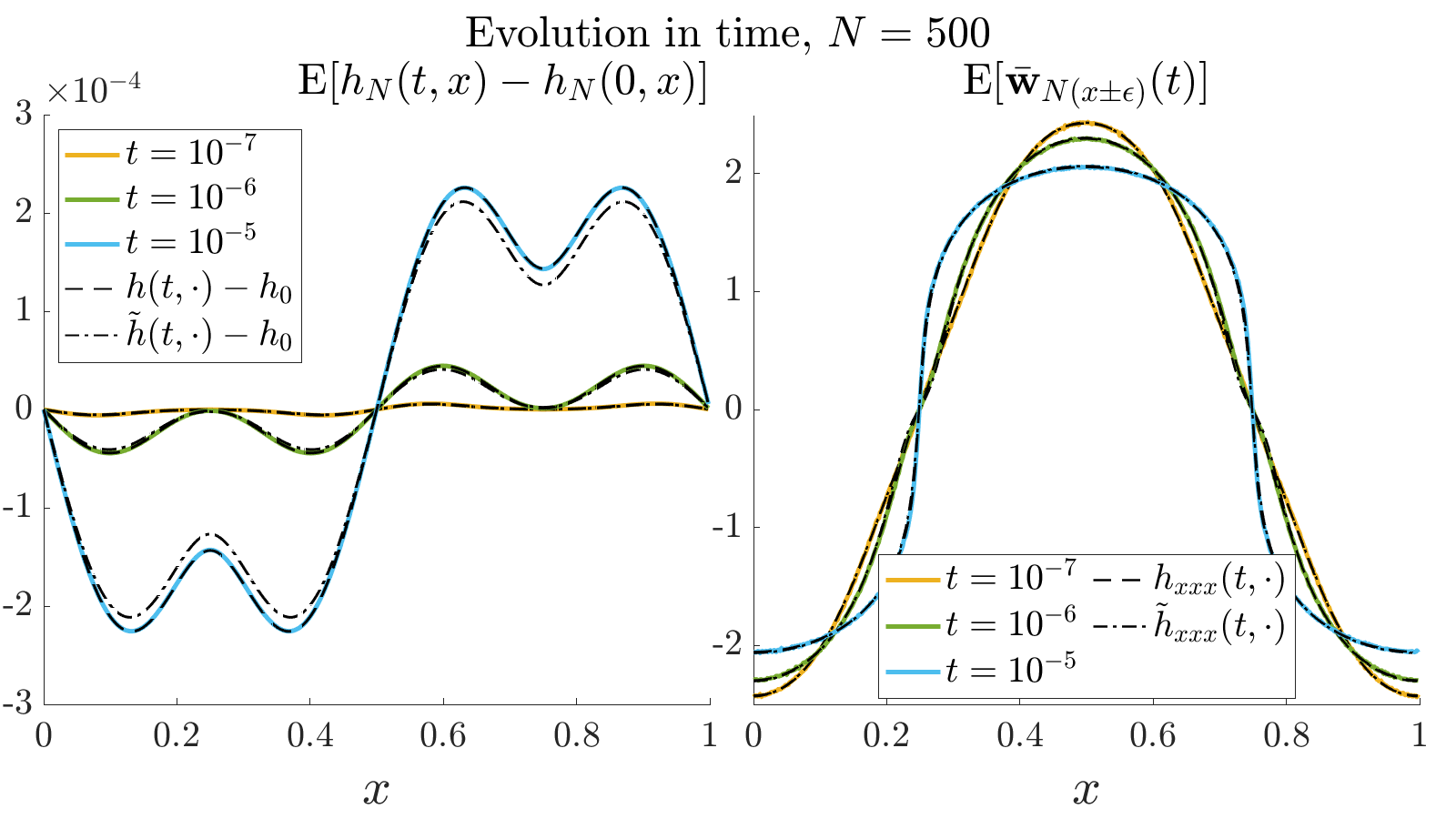}
\caption{Sinusoidal initial condition. Left: evolution of $\E[h_N(t,\cdot)-h_N(0,\cdot)]$ in time, for $N=500$. The evolution is compared to that of $h(t,\cdot)-h_0$ and $\tilde h(t,\cdot)-h_0$, where $h,\tilde h$ are the solutions to~\eqref{J-PDE} and~\eqref{tilde-J-PDE}, respectively. The discrepancy between $h(t,\cdot)-h_0$ and $\tilde h(t,\cdot)-h_0$ is clearest at $t=10^{-5}$. It is clear that $\E h_N$ follows $h$ rather than $\tilde h$. Right: evolution of $\E\varbar w(t)$ in time for $N=500$, compared to the evolution of $h_{xxx}$ and $\tilde h_{xxx}$. On the $O(1)$ scale of $h_{xxx}$, the difference between $h_{xxx}$ and $\tilde h_{xxx}$ is imperceptible.}
\label{fig:sin-evoln}
\end{subfigure}
\vspace{-5pt}
\begin{subfigure}[b]{\textwidth}
\includegraphics[width=\textwidth]{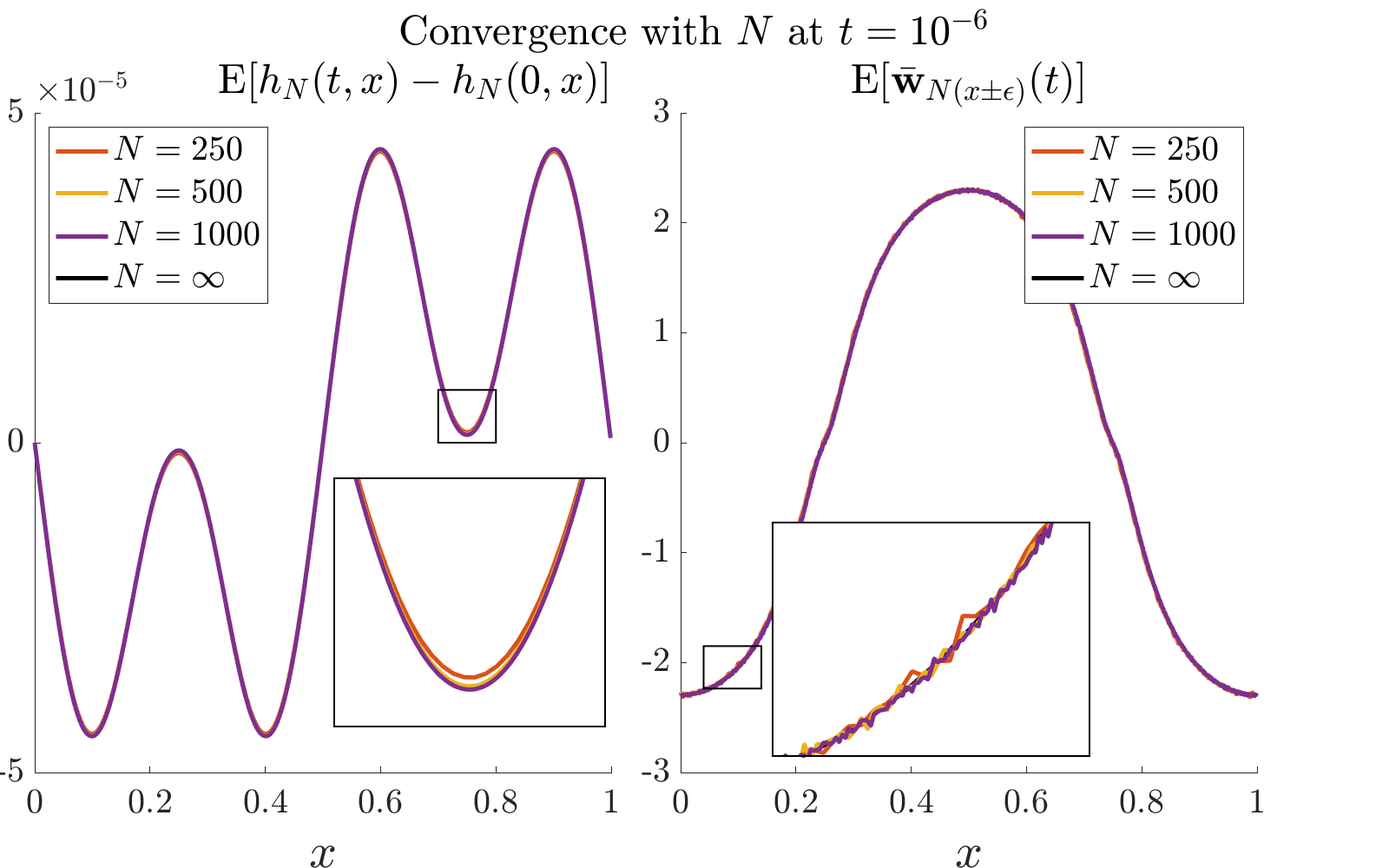}
\caption{Sinusoidal initial condition. Convergence of $\E[h_N(t,\cdot)-h_N(0,\cdot)]$ to $h(t,\cdot)-h_0$, and $\E\varbar w$ to $h_{xxx}$, as $N\to\infty$.}
\label{fig:sin-converge}
\end{subfigure}
\caption{}
\end{figure}
\section{PDE Analysis}\label{sec:PDE}
We conclude the paper by generalizing the PDE results in~\cite{gao2020}. Setting all constants equal to 1, consider the PDE
\beq\label{hPDE}
\begin{cases}
h_t = -\partial_{x}\l(\sigma(h_{xxx})\sinh(h_{xxx})\r),\quad &t>0, x\in\unit\\
h(0,x) = h_0(x),\quad &x\in\unit
\end{cases}
\eeq
where $\sigma\in C^2(\R)$, is even, nondecreasing on $\R^{+}$, and bounded above and below by constants $0<c<\sigma(\omega)<C$ for all $\omega\in\R$. These properties are all confirmed in Figure~\ref{fig:allUFOs}. From~\eqref{hPDE}, we get the following PDE for the slope $z=h_x$:
\beq\label{zPDE}
\begin{cases}
z_t = -\partial_{xx}\l(\sigma(z_{xx})\sinh(z_{xx})\r),\quad &t>0, x\in\unit\\
z(0,x) = z_0(x)=h_0'(x),\quad &x\in\unit
\end{cases}
\eeq
Before stating the main result, we introduce some notation. Let $$H=\{u\in L^2(\unit);\; \int_\unit udx=0\},\quad V=\{u\in H^2(\unit);\; \int_\unit udx=0\}.$$ Define $\psi:\R\to\R$ by
\beq
\psi(u) = c + \int_0^u\sigma(q)\sinh(q)dq,
\eeq 
and $\phi:H\to [0,+\infty]$ by
\beq\label{phi}
\phi(z) = \begin{cases}\int_\unit \psi(z_{xx})dx,\quad &z\in V,\\
+\infty,\quad&\text{otherwise}.\end{cases}
\eeq Note that we have $$\frac{\delta\phi}{\delta z} = \l(\psi'(z_{xx})\r)_{xx}= \partial_{xx}\l(\sigma(z_{xx})\sinh(z_{xx})\r)$$ so that~\eqref{zPDE} can be written as
$$z_t = -\frac{\delta\phi}{\delta z}.$$ This motivates writing solutions of~\eqref{zPDE} as the limit of a discretized gradient flow in the metric space $H$ with $L^2$ distance. 

In preparation for doing so, we state the following lemma. It is the same as Proposition 3.2 in~\cite{gao2020}, but applies to the more general functional $\phi$ in~\eqref{phi}:
\begin{lemma}\label{lma:phi-prop}
The functional $\phi:H\to[0,\infty]$ is $\lambda$-convex for $\lambda = c/\kappa^2$, where $\kappa$ is the best Poincare constant for the domain $\unit$. $\phi$ is also proper, lower semicontinuous in $H$, and satisfies coercivity, meaning that there exists a ball $B(u^*,r^*)=\{v\in H: \|v-u^*\|_{L^2}\leq r^*\}$ such that $\phi(u^*)<\infty$ and the infimum of $\phi$ over $B(u^*,r^*)$ is finite. 
\end{lemma}
See the end of this section for the proof. Now, define the proximal operator
$$\J_\tau[u] = \arg\min_{v\in H}\l\{\phi(v) + \frac{1}{2\tau}\|v-u\|^2\r\}.$$ The proximal operator is the variational formulation of the update for gradient descent on $\phi$ with step size $\tau$. The convexity and lower semicontinuity of $\phi$ ensures that the minimizer of the above objective exists and is unique. Using $\J_\tau$, we form the approximate solution
$$z_n(t) := \l(\J_{t/n}\r)^n[z_0].$$
Using Lemma~\ref{lma:phi-prop} and the theory of gradient flows in metric spaces (see~\cite{gao2020} and the citations therein, in particular~\cite{AGS}), one can show that given $z_0\in H$, the sequence $z_n(t)$ converges in $H$ to $z(t)$, which is the unique evolution variational inequality (EVI) solution to the PDE~\eqref{zPDE}. See~\cite{gao2020} and~\cite{AGS} for the definition of the EVI solution. Finally, if $h_0$ enjoys more regularity, then the EVI solution $z(t)$ is a strong solution. We have the following theorem, which is analogous to Theorem 3.6 in~\cite{gao2020}.
\begin{theorem}
Let $$D = \{z\in V\mid \l(\sigma(z_{xx})\sinh(z_{xx})\r)_{xx}\in H\}.$$ Take $T>0$ and $z_0\in D$ such that $\phi(z_0)<\infty$. Then~\eqref{zPDE} has a unique global strong solution $z$ in the sense that $\partial_tz = -\partial_{xx}\l(\sigma(z_{xx})\sinh(z_{xx})\r)$ for all $t\geq0$, and such that
$$z\in C([0,T]; D)\cap C_1([0,T]; H).$$ Moreover, we have the following decay:
\beqs
\|z(t)\|_{L^2}&\leq \|z_0\|_{L^2}\;\forall t\geq 0,\\
 \|\partial_tz(t)\|_{L^2} &\leq e^{-\lambda t}\l\|\partial_{xx}\l(\sigma\l(\partial_{xx}z_0\r)\sinh(\partial_{xx}z_0)\r)\r\|_{L^2}\forall t\geq0,
\eeqs
where $\lambda$ is as in Lemma~\ref{lma:phi-prop}.
\end{theorem}
Let us now present the proof of Lemma~\ref{lma:phi-prop}.
\begin{proof}[Proof of Lemma~\ref{lma:phi-prop}]
$\phi$ is proper since $u=0$ satisfies $\phi(u)<\infty$, so $\{\phi<\infty\}$ is nonempty. Since $\phi\geq0$, it is obviously coercive. Now we show $\phi$ is $\lambda$-convex with $\lambda = c/\kappa^2$, where $\kappa$ is the best Poincare constant for the domain $\unit$. First, note that $$\psi''(w) = \sigma(w)\cosh(w) + \sigma'(w)\sinh(w) \geq c\cosh(w)\geq c.$$ Now, analogously to~\cite{gao2020}, define
\beq
I(t) = \int_\unit (1-t)\psi(u_{xx}) + t\psi(v_{xx}) - \frac\lambda2t(1-t)\|u-v\|_{L^2}-\psi((1-t)u_{xx}+tv_{xx})dx.
\eeq Note that $I(0)=I(1)=0$, so $I(t)\geq0$ provided $I''(t)\leq 0$. We compute $I''(t)$ below, substituting $\lambda=c/\kappa^2$ in the second line:
\beqs
I''(t) &= \lambda\int_\unit(u-v)^2dx -\int_\unit(u_{xx}-v_{xx})^2\psi''((1-t)u_{xx}+tv_{xx})dx \\
&\leq  \frac{c}{\kappa^2}\int_\unit(u-v)^2dx - c\int_\unit(u_{xx}-v_{xx})^2dx\leq 0,
\eeqs
applying the Poincare inequality twice. Hence $\phi$ is $\lambda$-convex. The lower semi-continuity of $\phi$ will follow from the convexity of $\psi$ and the below bound~\eqref{L2bd}; for the details, see~\cite{gao2020}.  For $z\in V$, we have
\beqs
\frac c2\int_\unit\l|(z_{xx})^+\r|^2dx &\leq c\int_{\unit\cap \{z_{xx}>0\}} e^{(z_{xx})^+}dx \leq 2c\int_{\unit\cap\{z_{xx}>0\}}\cosh((z_{xx})^+)dx \\
&\leq 2c\int_{\unit}\cosh(z_{xx})dx\leq 2\phi(z).
\eeqs
Applying an analogous inequality with the negative part of $z_{xx}$, we conclude that
\beq\label{L2bd}
\|z_{xx}\|_{L^2}^2 \leq \frac8c\phi(z).
\eeq
\end{proof}

\section{Conclusion}\label{sec:conclude}
We have derived the continuity equation $h_t = -\partial_x\hat J(h_{xxx})$ governing the hydrodynamic limit of a Metropolis rate jump process in the rough scaling regime. Due to the surprising fact that the local equilibrium (LE) state of this process is not local Gibbs, and is unknown, we opted for a numerical approach to compute the current $\hat J$. We conclude with an observation about this approach. Although it took into account some specific properties of the model, the basic principle underlying our approach is quite general. Namely, if a system is in LE, then the expectations of a local nonlinear observable $f$ in different mesoscopic regions depend in the same way (through a universal function $\hat f$) on a finite number of usually linear statistics in these regions. In our case, this statistic is the first moment $\E\bar w_{\idxsetx}$, which is essentially the local value of $h_{xxx}$. We can infer the function $\hat f$ by plotting the $f$ expectations against the linear statistics, collected from sample runs of the process. We believe our numerical approach can be useful to derive the PDE limit of interacting particle systems in which an explicit expression for the LE distribution is not available, provided the PDE derivation reduces to computing the expectation of an observable $f$ in LE.
\appendix
\section{Proofs of Claims in Section~\ref{subsec:theory}}\label{app:claims}
The proof of Claim~\ref{claim:htow} relies on the following lemma.
\begin{lemma}\label{app:lemma:htow}
Under the conditions of the claim, there exists $R(t)>0$ such that 
$$\p\l(\frac1N\sum_{i=1}^NN^{-3}|h_i(t)|>R(t)\r)\to0,\quad N\to\infty.$$ \end{lemma}
\begin{proof} Let $C(t)=\sup_{N}\max_{i=1,\dots,N}\E|w_i(t)|$, which is finite thanks to~\eqref{bd1}. Let $h_i=h_i(t)$, $w_i=w_i(t)$. Write
\beq\label{hfromw}
h_i = \sum_{j=0}^{i-2}\sum_{k=0}^j\sum_{\ell=0}^kw_\ell + a_Ni^2 + b_Ni + c_N,\quad i=0,\dots, N-1,
\eeq
and note that we have the bound
\beqs
N^{-4}\sum_{i=1}^N|h_i| &\leq DN^{-4}\l[N^3\sum_{i=1}^N|w_i| + N^3|a_N| + N^2|b_N| + N|c_N|\r]\\
&= D\l[N^{-1}\sum_{i=1}^N|w_i| + N^{-1}|a_N| + N^{-2}|b_N| + N^{-3}|c_N|\r]\eeqs
for some constant $D$. Thus it suffices to show there exists a constant $R(t)$ such that $\p(N^{-1}\sum_{i=1}^N|w_i|>R(t)/4)$, $\p(N^{-1}|a_N|>R(t)/4)$, $\p(N^{-2}|b_N|>R(t)/4)$, and $\p(N^{-3}|c_N|>R(t)/4)$ all go to zero as $N\to\infty$. The first probability goes to zero by~\eqref{bd1}.

We will now solve for $a_N$, $b_N$, $c_N$. Note that taking $h_i$ as in~\eqref{hfromw}, we immediately get that $w_i=h_{i+2}-3h_{i+1}-3h_i + h_{i-1}$ for $i=1,\dots, N-3$, but we must also ensure that $w_i=h_{i+2}-3h_{i+1}-3h_i + h_{i-1}$ for $i=0,N-2,N-1$. This is equivalent to extending the definition of $h_i$ to $i=N, N+1, N+2$ and setting $h_0 =h_N$, $h_1=h_{N+1}$, $h_2=h_{N+2}$. One can show that the equality $h_2=h_{N+2}$ will follow from the other two equalities. Setting $h_0$ equal to $h_N$ gives
$$c_N = S_{N,3} + N^2a_N + Nb_N + c_N,\quad S_{N,3} = \sum_{j=0}^{N-2}\sum_{k=0}^j\sum_{\ell=0}^kw_\ell.$$ Setting $h_1$ equal to $h_{N+1}$ gives 
$$a_N + b_N + c_N = S_{N,3} + S_{N,2} + (N+1)^2a_N + (N+1)b_N + c_N,\quad S_{N,2} = \sum_{k=0}^{N-1}\sum_{\ell=0}^k w_\ell.$$
These two equations give $a_N = -S_{N,2}/2N$ and $b_N = S_{N,2} /2 - S_{N,3}/N$. It is straightforward to see that we have the bound $\E|S_{N,2}| \leq DC(t)N^2$ for some constant $D$, so $N^{-1}\E|a_N|= N^{-2}\E|S_{N,2}|/2 \leq DC(t)$. A similar argument gives $N^{-2}\E|b_N|\leq DC(t)$. 
We can estimate $N^{-3}|c_N|$ by recalling that $N^{-4}\sum_ih_i=:M_N$ converges to $M$ in probability. This gives
$$N^4M_N = S_{N,4} + a_N\sum_{i=0}^{N-1}i^2 +  b_N\sum_{i=0}^{N-1}i  + Nc_N,\quad S_{N,4} = \sum_{i=0}^{N-1}\sum_{j=0}^{i-2}\sum_{k=0}^j\sum_{\ell=0}^kw_\ell$$ so that
 $$N^{-3}|c_N|\leq |M_N| + N^{-4}|S_{N,4}| + DN^{-1}|a_N| + DN^{-2}|b_N|.$$ The first summand on the right is bounded in probability, and the second, thirds, and fourth summands are bounded in expectation, so it follows that $N^{-3}|c_N|$ is bounded in probability. 
\end{proof}
\begin{proof}[Proof of Claim~\ref{claim:htow}]
Let us show a unique function $h=h(t,\cdot)$ exists such that $\int_0^1hdx=M$, $h_{xxx}=w$, and $h,h_x,h_{xx}$ are all periodic. Such a function necessarily takes the form $$h(t ,x) = \int_0^x\int_0^y\int_0^zw(t,u)dudzdy + a(t)x^2 + b(t)x + c(t),$$ so we show there is a unique choice of $a(t),b(t),c(t)$. First note that by definition of $\wN$ as the third order FD of some process, we have $\sum_i w_i(t) =0$ for all $t$ (recall that lattice site indexing is periodic). Therefore, taking $\phi\equiv 1$, we get that $\int w(t,x)dx=0$ for all $t$. Now, we have $h_{xx} = \int_0^xw(t,u)du + 2a(t)$, which is periodic for any $a(t)$, since $\int w(t,x)dx=0$. Equating $h(t,0)$ and $h(t,1)$, we get the condition
$$c(t) =h(t,0)=h(t,1)= \int_0^1\int_0^y\int_0^zw(t,u)dudzdy + a(t) + b(t)+c(t).$$ Equating $h_x(t,0)$ with $h_x(t,1)$, we get the condition
$$b(t) = h_x(t,0)=h_x(t,1)=\int_0^1\int_0^zw(t,u)dudz + 2a(t) + b(t).$$ Finally, integrating $h$, we get the condition
$$M = \int_0^1h(x)dx=\int_0^1\int_0^x\int_0^y\int_0^zw(t,u)dudzdydx + a(t)/3 + b(t)/2+c(t).$$ It is clear that this system of equations has a unique solution $a(t),b(t),c(t)$, so a unique $h$ satisfying the conditions exists. Now, we need to show that for all $\phi\in\cts$, we have \beq\label{needtoshow}N^{-1}\sum_i\phi(i/N)N^{-3}h_i(t)\probto\int_0^1\phi(x)h(t,x)dx\eeq for this $h$. Since $\sum_ih_i(t)$ stays fixed under the crystal surface dynamics, we already know this is true for $\phi\equiv 1$. Indeed, we have $$\frac1N\sum_{i=1}^NN^{-3}h_i(t) = \frac1N\sum_{i=1}^NN^{-3}h_i(0) \to M=\int_0^1h(t,x)dx.$$ Thus, it suffices to show~\eqref{needtoshow} for continuous $\phi$ which integrate to $0$. For such a $\phi$, there exists a $C^3$, periodic $\psi$ such that $\psi'''=\phi$ and $\psi'$, $\psi''$ are also periodic. This is true by the same argument as above. Now, let $\psi_i = \psi(i/N)$, and
\beq\psi_i^1= \psi_{i-1}-\psi_{i-2},\quad\psi_i^2 = \psi_{i+1}^1 - \psi_i^1,\quad \psi_i^3 = \psi_{i+1}^2 - \psi_i^2.\eeq
Note that by continuity of $\phi=\psi'''$, 
$$C_N :=  \max_i\l|\psi'''(i/N) - N^3\psi_i^3\r| \to0,\quad N\to\infty.$$ We then have
\beqs\label{by-parts}\bigg|\frac1N\sum_{i=1}^N&\psi'''(i/N)N^{-3}h_i-\frac1N\sum_{i=1}^N\psi_i^3h_i\bigg|
\\&\leq C_N \frac1N\sum_{i=1}^NN^{-3}|h_i|,\eeqs omitting the $t$ for brevity. Therefore,
\beqs
\p\bigg(\bigg|\frac1N\sum_{i=1}^N\psi'''&(i/N)N^{-3}h_i-\frac1N\sum_{i=1}^N\psi_i^3h_i\bigg|>\delta\bigg)\\
&\leq \p\bigg(\frac1N\sum_{i=1}^NN^{-3}|h_i|>\delta/C_N\bigg)\to0,\quad N\to\infty,\eeqs using the Lemma. Thus it suffices to prove $\frac1N\sum_{i=1}^N\psi_i^3h_i$ converges in probability to $\int\phi(x)h(t,x)dx$. Define 
\beq h_i^1 = h_i - h_{i-1},\quad h_i^2 = h_i^1 - h_{i-1}^1,\quad h_i^3 = h_i^2 - h_{i-1}^2\eeq  Note that $h_i^3= h_{i}-3h_{i-1} + 3h_{i-2}-h_{i-3} = w_{i-2}.$ Now, for arbitrary $N$-periodic sequences $\{f_k\}_{k\in\Z}$, $\{g_k\}_{k\in\Z}$, we have by the summation by parts formula,
$$\sum_{k=1}^{N}f_k(g_{k+1}-g_k) = f_{N}g_{N+1} - f_1g_1 - \sum_{k=2}^{N}g_k(f_k - f_{k-1}) = -\sum_{k=1}^{N}g_k(f_k - f_{k-1}),$$ so there are no boundary terms thanks to the periodicity. We now apply summation by parts three times to get
\beqs
\sum_{i=1}^N\psi_i^3h_i &=  -\sum \psi_i^2h_i^1 = \sum_{i=1}^N\psi_i^1h_i^2 = \sum_{i=1}^N(\psi_{i-1}-\psi_{i-2})h_i^2 \\
&= -\sum_{i=1}^N\psi_{i-2}h_i^3 = -\sum_{i=1}^N\psi_{i-2}w_{i-2} = -\sum_{i=1}^N\psi(i/N)w_i
\eeqs
Thus $$\frac1N\sum_{i=1}^N\psi_i^3h_i = -\frac1N\sum_{i=1}^N\psi(i/N)w_i\probto - \int_0^1\psi(x)w(t,x)dx = \int_0^1\psi'''(x)h(t,x)dx.$$
The last equality is by three applications of integration by parts. There are no boundary terms because $\psi$, $h$, and their first three spatial derivatives, are all periodic. 
\end{proof}
For the proof of Claims~\ref{claim:meso},~\ref{claim:PDE}, recall that to a vector $\mbf v = (v_1,\dots, v_N)$ we associate a signed measure on the unit interval, as follows:
\beq\label{v-meas}\mbf v \quad\leftrightarrow\quad v(dx) = \frac1N\sum_{i=1}^Nv_i\delta\l(x-\frac iN\r).\eeq
Also, recall from Remark~\ref{rk:w-meas} that $(\phi\ast\mu)(x) = \int_0^1 \phi(x-y)\mu(dy)$ for a signed measure $\mu$ defined on the unit torus and a function $\phi\in L^1(\mu)$, and that $(\phi_\epsilon\ast w_N(t,\cdot))(x) = \varbar w(t),$ where $\phi_\epsilon(x)= \frac{1}{2\epsilon}\mathbbm{1}_{(-\epsilon,\epsilon)}(x)$. Further, note that if $\phi$ is even, and the function $(x,y)\mapsto \psi(x)\phi(x-y)$ is integrable with respect to $\mu(dy)dx$ on $\unit\times\unit$, then we have the identity
\beqs\label{conv-meas}
\int_0^1 \psi(x)(\phi\ast\mu)(x)dx &= \int_0^1\int_0^1\psi(x)\phi(x-y)\mu(dy)dx \\&=  \int_0^1\mu(dy)\int_0^1\psi(x)\phi(y-x)dx
= \int_0^1(\psi\ast\phi)(y)\mu(dy).
\eeqs

\begin{proof}[Proof of Claim~\ref{claim:meso}]
Since $\psi$ is continuous and hence uniformly continuous on $[0,1]$, we have $\sup_{x\in\unit}|\psi(x) -(\psi\ast\phi_\epsilon)(x)|\to0$ as $\epsilon\to0$, with $\phi_\epsilon$ as above. Using this and~\eqref{bd1}, we have that 
\beqs
\E\big|\int \psi(x)w_N(t,dx) &- \int(\psi\ast\phi_\epsilon)(x)w_N(t,dx)\big| \\
&\leq \max_i|(\psi-\psi\ast\phi_\epsilon)(i/N)|\max_i\E|w_i|
\eeqs goes to zero as $N\to\infty$ and then $\epsilon\to0$. Therefore, it suffices to show $\int(\psi\ast\phi_\epsilon)(x)w_N(t,dx)$ converges in $L^1$ (with respect to randomness) to $\int\psi(x)w(t,x)dx$. Now, $\int(\psi\ast\phi_\epsilon)(x)w_N(t,dx) = \int \psi(x)\varbar w(t)dx$ by~\eqref{conv-meas}, and 
\beq
\E\bigg|\int \psi(x)\varbar w(t)dx - \int\psi(x)w(t,x)dx\bigg|\leq \|\psi\|_{\infty}\int\E|\varbar w(t) - w(t,x)|dx,
\eeq
which goes to zero by the definition of pointwise mesoscopic convergence, combined with~\eqref{bd1} and the continuity of $w$ which allows us to apply Lebesgue Dominated Convergence.
\end{proof}
\begin{proof}[Proof of Claim~\ref{claim:PDE}]
As argued in the main text, the lefthand side of~\eqref{weak-def} is the limit of $\int_\unit\psi(x)\E\l[\varbar w(t)-\varbar w(0)\r]dx$
and
\beqsn
\int_0^1\psi(x)&\E[\varbar w(t)-\varbar w(0)]dx\\&=  N^4\int_0^t\frac1N\sum_{i=1}^N(\psi\ast\phi_\epsilon)\l(\frac iN\r)\E[D_N^4J(w_i(s))]ds,
\eeqsn
where $D_N^4J(w_i) = J(w_{i-2})-4J(w_{i-1}) +6J(w_i)- 4J(w_{i+1}) + J(w_{i+2})$. We can now use that summation by parts yields no boundary terms when the sequences are periodic, as above. Thus we can move $D_N^4$ onto $(\psi\ast\phi_\epsilon)(i/N)$ provided we define $D_N^4(\psi\ast\phi_\epsilon)(i/N)$ as the appropriately shifted fourth order finite difference obtained in the summation by parts, rather than the centered fourth order FD. Thus we can write
\beqsn
\int_0^1\psi(x)\E[&\varbar w(t)-\varbar w(0)]dx\\&=  N^4\int_0^t\frac1N\sum_{i=1}^ND_N^4(\psi\ast\phi_\epsilon)\l(\frac iN\r)\E[J(w_i(s))]ds.
\eeqsn
Since $D_N^4$ is only shifted by a finite number of indices, we still have by the smoothness of $\psi$ that
$$\l|N^4D_N^4(\psi\ast\phi_\epsilon)(i/N)-(\psi^{(4)}\ast\phi_\epsilon)(i/N)\r|\leq \frac CN\|\psi^{(5)}\|_\infty\|\phi_\epsilon\|_{L^1} = \frac CN\|\psi^{(5)}\|_\infty$$ for some constant $C$. Thus,
\beqs
\bigg|\int_0^t\frac1N\sum_{i=1}^N&\l[N^4D_N^4(\psi\ast\phi_\epsilon)-(\psi^{(4)}\ast\phi_\epsilon)\r]\l(\frac iN\r)\E J(w_i(s))ds\bigg|\\
&\leq \frac CN\int_0^t\max_i|\E J(w_i(s))|ds,\eeqs which goes to zero as $N\to\infty$ by the boundedness assumption~\eqref{bd2}. Next, we have
\beqs
\int_0^t&\frac1N\sum_{i=1}^N(\psi^{(4)}\ast\phi_\epsilon)(i/N)\E J(w_i(s))ds \\&= \int_0^t\int_0^1(\psi^{(4)}\ast\phi_\epsilon)(x)\E J(\wN(s))(dx)ds=\int_0^t\int_0^1\psi^{(4)}(x)\E\bar J({\bf w}_{\idxsetx}(s)) dxds,
\eeqs using identity~\eqref{conv-meas} with $\mu(dx) = \E J(\wN(s))(dx)$ $=$ $\frac1N\sum_{i=1}^N\E J(w_i(s))\delta(x-i/N)$ and $\phi=\phi_\epsilon$. By the pointwise convergence of $\E\bar J({\bf w}_{\idxsetx}(s))$ to $\hatJ(w(s,x))$ and boundedness~\eqref{bd2}, we conclude by applying dominated convergence. 
\end{proof}
\section{Justification of Time Averaging}\label{app:num}
We now justify using a sample average estimate of the time average $\E_\Delta[f(w_i(t))]:=\frac1\Delta\int_{I_{t,\Delta}}\E[f(w_i(s))]ds$ in place of a sample average estimate of $\E[f(w_i(t))]$. Let $\E^n$ and $\E^n_\Delta$ denote the $n$-sample estimates of $\E$ and $\E_\Delta$, respectively (see~\eqref{sample-av} and~\eqref{sample-time-av}). We first show that by taking $\Delta$ small enough, decreasing $\Delta$ further has no effect on $\E^n_\Delta[f( w_i)]$, except perhaps to increase its variance. This is shown in the left panels in Figure~\ref{fig:sample-v-time} (a), (b) for $f(w)=J(w)$ and $f(w)=w$, respectively. Fixing $\Delta=2\times 10^{-9}$, we now show that as we increase $n$, the estimate $\E^n[f(w_i)]$ approaches $\E^n_\Delta[f(w_i)]$. See the righthand panels in Figure~\ref{fig:sample-v-time} (a), (b).
\begin{figure}
\centering
\begin{subfigure}[b]{\textwidth}
\includegraphics[width=\textwidth]{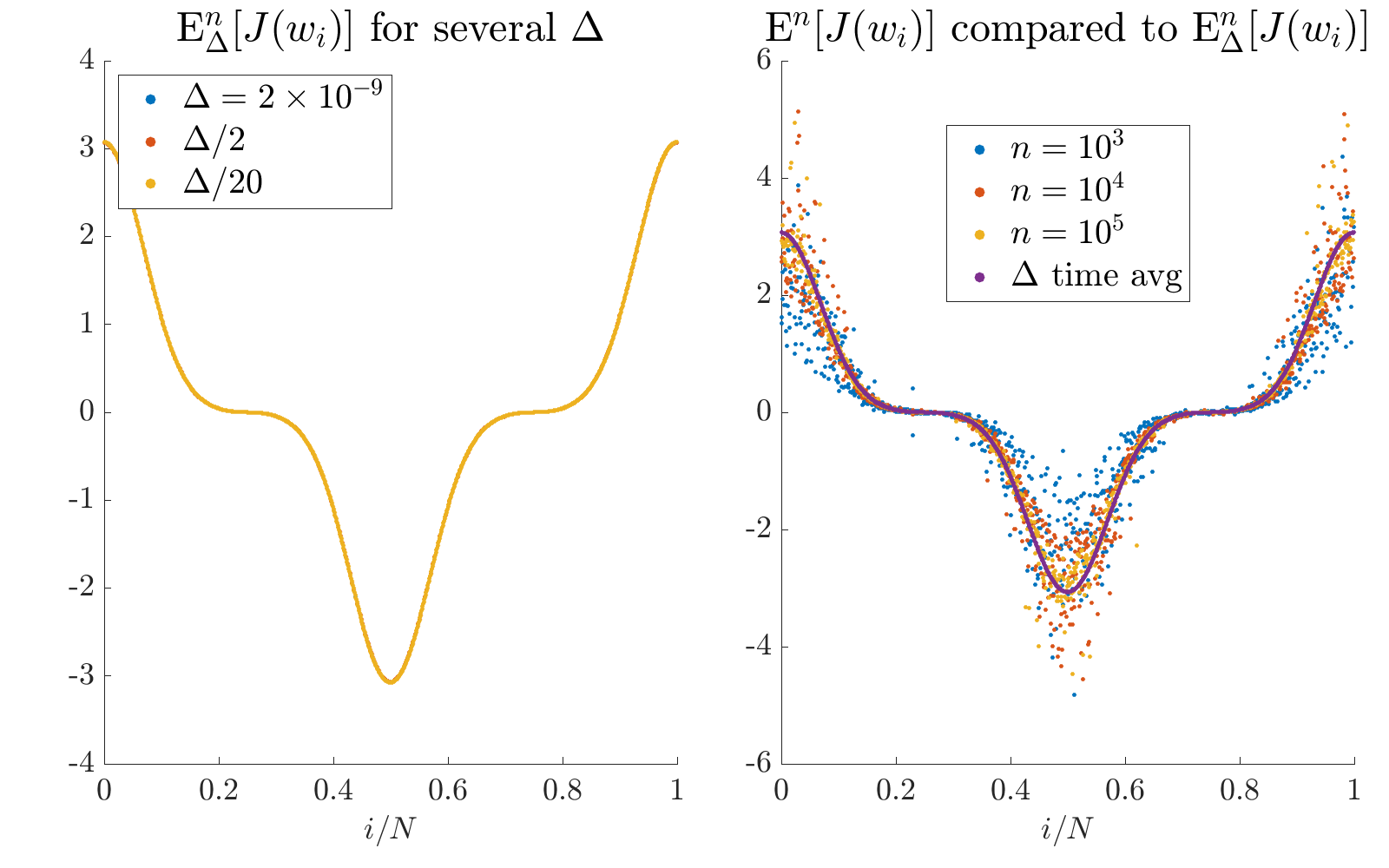}
\caption{Left: we choose $\Delta$ small enough, so that the effect of decreasing $\Delta$ on the estimator $\E^n_\Delta[J( w_i)]$ is negligible. Right: As $n$ increases, the instantaneous-time estimator $\E^n[J( w_i)]$ approaches $\E_\Delta[J( w_i)]$.}
\end{subfigure}
\begin{subfigure}[b]{\textwidth}
\includegraphics[width=\textwidth]{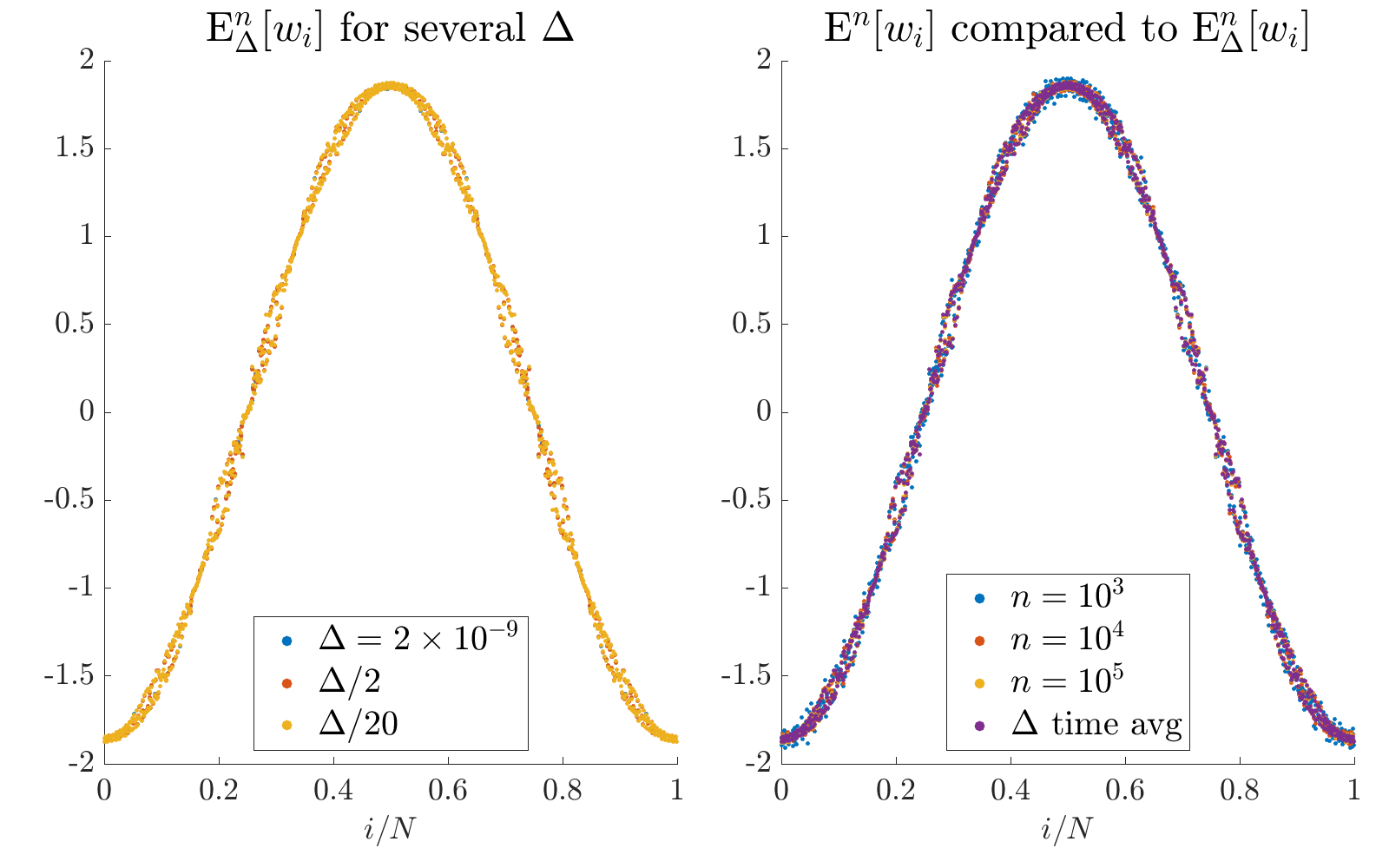}
\caption{Same as above, but now with the observable $\wN\mapsto w_i$.}
\end{subfigure}
\caption{}
\label{fig:sample-v-time}
\end{figure}

\bibliographystyle{alpha}
\bibliography{bibliogr}

\newcommand{\etalchar}[1]{$^{#1}$}
\begin{thebibliography}{LLMM19}

\bibitem[AGS08]{AGS}
Luigi Ambrosio, Nicola Gigli, and Giuseppe Savar{\'e}.
\newblock {\em Gradient flows: in metric spaces and in the space of probability
  measures}.
\newblock Springer Science \& Business Media, 2008.

\bibitem[EDZR18]{Embacher2018}
P~Embacher, N~Dirr, J~Zimmer, and C~Reina.
\newblock Computing diffusivities from particle models out of equilibrium.
\newblock {\em Proc. R. Soc. A}, 474, 2018.

\bibitem[FS97]{funaki1997motion}
Tadahisa Funaki and Herbert Spohn.
\newblock Motion by mean curvature from the ginzburg-landau interface model.
\newblock {\em Communications in Mathematical Physics}, 185(1):1--36, 1997.

\bibitem[Gil76]{KMC}
Daniel~T Gillespie.
\newblock A general method for numerically simulating the stochastic time
  evolution of coupled chemical reactions.
\newblock {\em Journal of Computational Physics}, 22(4):403--434, 1976.

\bibitem[GKL{\etalchar{+}}20]{gao2020}
Yuan Gao, Anya~E. Katsevich, Jian-Guo Liu, Jianfeng Lu, and Jeremy~L. Marzuola.
\newblock Analysis of a fourth order exponential pde arising from a crystal
  surface jump process with metropolis-type transition rates.
\newblock {\em Pure and Applied Analysis}, 3(4), 2020.

\bibitem[GLLM20]{gao2020_arrPDE}
Yuan Gao, Jian-Guo Liu, Jianfeng Lu, and Jeremy~L Marzuola.
\newblock Analysis of a continuum theory for broken bond crystal surface models
  with evaporation and deposition effects.
\newblock {\em Nonlinearity}, 33(8):3816--3845, jun 2020.

\bibitem[GPV88]{varadhanI}
M.~Z. Guo, G.~C. Papanicolaou, and S.~R.~S. Varadhan.
\newblock {Nonlinear diffusion limit for a system with nearest neighbor
  interactions}.
\newblock {\em Communications in Mathematical Physics}, 118(1):31 -- 59, 1988.

\bibitem[Kat21]{kat21}
Anya Katsevich.
\newblock The local equilibrium state of a crystal surface jump process in the
  rough scaling regime.
\newblock {\em arXiv preprint arxiv:2106.04652}, 2021.

\bibitem[KDM95]{krug1995adatom}
J~Krug, HT~Dobbs, and S~Majaniemi.
\newblock Adatom mobility for the solid-on-solid model.
\newblock {\em Zeitschrift f{\"u}r Physik B Condensed Matter}, 97(2):281--291,
  1995.

\bibitem[KL98]{kipnisbook}
Claude Kipnis and Claudio Landim.
\newblock {\em Scaling limits of interacting particle systems}, volume 320.
\newblock Springer Science \& Business Media, 1998.

\bibitem[LLMM19]{asymmetry}
Jian-Guo Liu, Jianfeng Lu, Dionisios Margetis, and Jeremy~L. Marzuola.
\newblock Asymmetry in crystal facet dynamics of homoepitaxy by a continuum
  model.
\newblock {\em Physica D: Nonlinear Phenomena}, 393:54--67, 2019.

\bibitem[Man13]{statphys}
F.~Mandl.
\newblock {\em Statistical Physics}.
\newblock Manchester Physics Series. Wiley, 2013.

\bibitem[MW13]{mw-krug}
Jeremy~L. Marzuola and Jonathan Weare.
\newblock Relaxation of a family of broken-bond crystal-surface models.
\newblock {\em Phys. Rev. E}, 88:032403, Sep 2013.

\bibitem[Nis02]{Nishikawa}
Takao Nishikawa.
\newblock Hydrodynamic limit for the ginzburg-landau $\varphi$ interface model
  with a conservation law.
\newblock {\em Journal of Mathematical Sciences. The University of Tokyo},
  9:481--519, 2002.

\bibitem[Spo12]{spohnbook}
Herbert Spohn.
\newblock {\em Large scale dynamics of interacting particles}.
\newblock Springer Science \& Business Media, 2012.

\end{thebibliography}
\vspace{1cm}

\end{document}